\documentclass[a4paper,12pt]{article}
\makeatletter
\@addtoreset{footnote}{page}
\makeatother

\usepackage{enumitem}

\usepackage{amsthm}
\usepackage{amsmath,amssymb,latexsym,amsfonts,mathrsfs}

\renewcommand{\Re}{\mathop{\rm Re}}

\renewcommand{\tilde}{\widetilde}
\renewcommand{\bar}{\overline}

\newcommand{\absol}[1]{\left| #1 \right|} %absolute value
\newcommand{\rbra}[1]{\!\left( #1 \right)} %round brackets or parentheses
\newcommand{\sbra}[1]{\!\left[ #1 \right]} %square brackets or brackets
\newcommand{\bC}{\ensuremath{\mathbb{C}}}
\newcommand{\bD}{\ensuremath{\mathbb{D}}}
\newcommand{\bE}{\ensuremath{\mathbb{E}}}

\newcommand{\bN}{\ensuremath{\mathbb{N}}}

\newcommand{\bP}{\ensuremath{\mathbb{P}}}

\newcommand{\bR}{\ensuremath{\mathbb{R}}}

\newcommand{\cB}{\ensuremath{\mathcal{B}}}

\newcommand{\cF}{\ensuremath{\mathcal{F}}}

\newcommand{\cQ}{\ensuremath{\mathcal{Q}}}

\theoremstyle{plain}
\newtheorem{Thm}{Theorem}[section]

\newtheorem{Lem}[Thm]{Lemma}
\newtheorem{Prop}[Thm]{Proposition}
\newtheorem{Cor}[Thm]{Corollary}

\theoremstyle{definition}

\newtheorem{Def}[Thm]{Definition}
\newtheorem{Rem}[Thm]{Remark}

\numberwithin{equation}{section}

\makeatletter
\renewcommand\section{\@startsection {section}{1}{\z@}%
                                   {-3.5ex \@plus -1ex \@minus -.2ex}%
                                   {2.3ex \@plus.2ex}%
                                   {\normalfont\large\bf}}
\makeatother

\makeatletter
\renewcommand\subsection{\@startsection {subsection}{1}{\z@}%
                                   {-3.5ex \@plus -1ex \@minus -.2ex}%
                                   {2.3ex \@plus.2ex}%
                                   {\normalfont\normalsize\bf}}
\makeatother

\usepackage{mathtools}
\mathtoolsset{showonlyrefs=true}

\begin{document}
\begin{center}
	{\Large \bf
    Analytic property of scale functions for standard processes with no negative jumps and its application to quasi-stationary distributions
	}
\end{center}
\begin{center}
	Kei Noba and Kosuke Yamato
\end{center}
\begin{center}
	{\small \today}
\end{center}

\begin{abstract}
	For a scale function of standard processes with no negative jumps, we characterize it as a unique solution to a Volterra integral equation.
	This allows us to extend it to an entire function and to derive a useful identity that we call the resolvent identity.
	We apply this result to study the existence of a quasi-stationary distribution for the processes killed at hitting boundaries.
	A new classification of the boundary, which is a natural extension of Feller's for one-dimensional diffusions, is introduced and plays a central role to characterize the existence. 
\end{abstract}

\section{Introduction}

Scale functions have been widely used in the studies of one-dimensional diffusions and spectrally one-sided L\'evy processes (see, e.g., {\cite[Chapter VII]{RevuzYor} and \cite[Chapter 5]{Ito_essentials} for one-dimensional diffusions,} and \cite{MR3014147} and \cite[Section 8]{KyprianouText} for L\'evy processes).
%Scale functions are a family of functions that give simple formulas for the exit times and the potential density on finite intervals.
%[There is probably a problem with the above way of writing. For example, the function that appeared in the previous paper by Prof. Rivero and his students is defined a "scale function" although they did not show the potential density from their scale functions. I guess by p.100 in Levy matter we should change as follows:]
Scale functions are a family of functions that give simple formulas for the exit times on finite intervals. 
In many cases, the scale functions also give formulas for the potential density. 
These formulas %{\color{blue}{[Instead of "The formulas above", "These fomulas" may be better since the fomulas have not yet explicitly written.]}} 
allow many important functions and values related to the processes to be represented by the scale functions,
so that we can study them using the properties of the scale functions.

Since the approach by the scale functions has been very successful,
a number of studies have {extended them} to more general Markov processes with no negative jumps (see \cite{MR2641768, MR2548498, MR4025708, MR3944781}).
%In particular, the scale functions introduced in \cite{NobaGeneralizedScaleFunc} via an excursion measure are applicable to any standard processes with no negative jumps 
%[How about the following?]
In particular, \cite{NobaGeneralizedScaleFunc} defined the scale functions for very wide class of %{\color{blue} [I think "any" is not accurate, because we need several technical assumptions. How about "very wide class of"?]} 
standard processes with no negative jumps via an excursion measure and gave simple formulas for the potential densities killed on exiting finite intervals
%{[The sentence is overlapping with the one in the previous paragraph, so I think it is better to leave it as it is now.]}
%{[Looking at it again, I thought that "Scale functions are a family of functions that give simple formulas for the exit times and the potential density on finite intervals" in the previous paragraph should be changed. ]}
(see, e.g., \cite[Section 1.9]{BluGet1968} for standard processes).
The main objective of the present paper is to derive several useful properties of the scale functions given in \cite{NobaGeneralizedScaleFunc}.
As an application, we study the existence of quasi-stationary distributions (QSDs) of these processes killed at the boundaries.

\subsection{Analytic properties of the scale functions}

For a standard process $X$ with no negative jumps on an interval $I \subset \bR$, the $q$-scale function {introduced} in \cite{NobaGeneralizedScaleFunc} with $q \geq 0$ is given by a bivariate non-negative function $W^{(q)}(x,y)$ on $I^2$.
We will recall the definition and formulas in Section \ref{section:scaleFunc}. 

We investigate two new important properties of the scale functions. 
The first is analyticity of the scale functions.
Precisely, {the function} $W^{(q)}$ has the following series expansion with respect to $q$  %and can be regarded as an entire function {\color{blue}["can be regarded as an entire function" seems not necessary since it is mentioned in the next sentence.]} 
for fixed $x,y\in I $: 
\begin{align}
	W^{(q)}(x,y) = \sum_{n \geq 0}q^{n}W^{\otimes (n+1)}(x,y) \quad ({q \geq 0}, \ x,y \in I), \label{304}
\end{align}
where $W := W^{(0)}$ and $W^{\otimes n}$ denotes the $n$-th product of $W$ in a certain sense, {which is defined in \eqref{eq09}.}
%{[I think we should consider this equality only for $q \geq 0$ here since the LHS is extended to $q \in \bC$ based on this.]}
As a result, the function {$q\mapsto W^{(q)}(x,y)$} can be analytically extended to $\bC$.
The second is an identity {similar to the one satisfied by resolvents}:
\begin{align}
	W^{(q)}(x,y) - W^{(r)}(x,y) = (q-r) W^{(q)} \otimes W^{(r)}(x,y) \quad (q,r \in \bC, \ x,y \in I), \label{eq78}
\end{align}
where again $W^{(q)} \otimes W^{(r)}$ denotes a product of $W^{(q)}$ and $W^{(r)}$ in a certain sense. 
We call this the \textit{resolvent identity}.

{ %[How about rewriting the rest of this subsection as follows?]

If $X$ is a spectrally positive L\'evy process, these formulas are already known and can be easily shown.
In the theory of spectrally positive L\'evy processes, for $q \geq 0$, a unique function $\tilde{W}^{(q)}:\bR \to [0, \infty)$ with the following properties is called the $q$-scale function: it is zero on $(-\infty,0)$, and it is continuous on $[0,\infty)$ and satisfies the following equation:
\begin{align}
	\int_{0}^{\infty}\mathrm{e}^{-\beta t}\tilde{W}^{(q)}(x)dx = \frac{1}{\psi(\beta) - q} \quad \text{for large $\beta > 0$,} \label{eq77}
\end{align}
where $\psi$ denotes the Laplace exponent of $X$ (see, e.g., \cite[Theorem 8.1]{KyprianouText}).
We can take the reference measure of $X$ as the Lebesgue measure, and then this scale function $\tilde{W}^{(q)}$ is related to the bivariate scale function $W^{(q)}$ as $\tilde{W}^{(q)}(y-x)=W^{(q)}(x, y)$ (see, \cite[(1.10)]{NobaGeneralizedScaleFunc}).
In this case, the product $\otimes$ in \eqref{304} and \eqref{eq78} coincides with the usual convolution product of functions on $\bR$, and we can get \eqref{304} and \eqref{eq78} by simple computations because the Laplace transform of both sides can be computed explicitly using \eqref{eq77}.

Though these formulas are not difficult for spectrally positive L\'evy processes, it has played prominent roles. 
The identity \eqref{304} was observed in \cite[(9)]{BertoinQSD}, and 
he used it to show the positivity of the decay parameter and the unique existence of the QSD when the process is killed on exiting a finite interval.
The existence of QSDs in the half-line case was studied in \cite{QSD_SPL}, where the formulas were still the main tool.
Furthermore, it was applied to characterize the Laplace transform of the hitting distribution using the change of measure (see, e.g., \cite[Section 3.3]{MR3014147}). 
%On the resolvent identity, 
The resolvent identity \eqref{eq78} was obtained in \cite[(6)]{MR3148018}, and they applied it to obtain the occupation times of intervals until certain first passage times for spectrally positive L\'evy processes. 
The identity and its variants have been useful for expressing the expectations of various values such as risks, dividends and costs using the scale functions (see, e.g., \cite{MR4154775, MR4175396, MR4237145, MR4309476, MR4525666, MR4536183}).

In our generalized setting, however, there is no useful formula such as \eqref{eq77} that is essentially a consequence of the space-homogeneity.
We need a different approach.
Our proof is in a sense very elementary, which just combines the strong Markov property with the fundamental two formulas in the theory of scale functions, that is, those for the exit time and the potential density.
We mention \cite{QSD_downward_skip-free}, which defined scale functions for downward skip-free Markov chains on $\bN$ by an excursion measure following \cite{NobaGeneralizedScaleFunc}.
He showed that it is extended to an entire function and satisfy the identity corresponding to \eqref{304} and \eqref{eq78}.
These properties played a central role to study the existence of QSDs.
We can say that this paper establishes results similar to \cite{QSD_downward_skip-free} for standard processes with no negative jumps.
We emphasize, however, that in proofs, simple analogies often fail to work due to technical difficulties, and we require different discussions.
}

\subsection{Quasi-stationary distributions}

{%[How about rewriting this subsection as follows?]

A quasi-stationary distribution is a kind of equilibrium state for Markov processes with killing.
It is an initial distribution such that the distribution of the process conditioned to survive satisfies the stationarity.
Precisely, for the lifetime $\zeta$ of $X$ a distribution $\nu$ on $I$ is called a QSD when
\begin{align}
	\bP_{\nu}[X_{t} \in dx \mid \zeta > t] = \nu(dx) \quad \text{for every $t > 0$}. \label{}
\end{align}

The existence of QSDs has been studied by many authors, and there have been several sufficient conditions for the existence (see, e.g., \cite{KingmanRrecurrence,TuominenTweedie,RenewalDynamicalApproach,Takeda:QSD,ChampagnatVillemonais}).
On the other hand, there is a limited class of Markov processes for which a necessary and sufficient condition for the existence of QSDs has been obtained.
Such classes are birth-and-death processes (\cite{vanDoornQSD-BD}), one-dimensional diffusions (\cite[Chapter 6.3]{Quasi-stationary_distributions}, \cite{Littin}), spectrally positive L\'evy processes (\cite{BertoinQSD,QSD_SPL}) and branching processes (\cite{LambertQSD,MaillardQSD-branching}).

We characterize the existence for standard processes with no negative jumps when the killing occurs at boundaries.
Our process includes the various classes mentioned above, i.e., one-dimensional diffusions, spectrally positive L\'evy processes and continuous-state branching processes.
This means that our methods provide a unified approach to the results obtained in previous studies by the specific toolkits for each class.
We show that the existence of QSDs is completely characterized by the positivity of the decay parameter and the classification of the boundary, which we introduce via an integrability of the $0$-scale function.
It is worth noting that the new classification gives a natural extension of Feller's for one-dimensional diffusions.
We also prove that every QSD is represented by a $q$-scale function for negative $q$.
When the boundary is classified as entrance (see Definition \ref{def:boundaryClassification}), we investigate the \textit{Yaglom limit}, that is, the limit distribution of $\bP_{x}[X_{t} \in dx \mid \zeta > t]$ as $t \to \infty$.
Our approach is based on the scale function and is similar to that for spectral positive L\'evy processes and downward skip-free Markov chains.
In the proof, the analyticity of the scale function and the resolvent identity \eqref{eq78} play an essential role.
}

\subsection{Outline}

We present the outline of this paper.
In Section \ref{section:preliminary}, we will prepare some elements from the potential theory and define the scale functions. We also discuss their basic properties.
Section \ref{section:analyticExtension} is one of the main parts.
We will prove the analyticity and the resolvent identity of the scale functions.
In Section \ref{section:positivity}, we discuss the positivity of the scale functions.
This section is, to some extent, a preparation for the next section, but it is of interest in its own right and is expected to be useful in future applications.
Section \ref{section:QSD} is the other main part.
We will give a complete characterization of the existence of QSDs when the upper boundary is inaccessible.
The main theorems will be presented in the first part of the section.
In Appendix \ref{appendix:twoSideExit}, we will show the unique existence of a QSD when the upper boundary is {accessible}.
The proof is given by a slight modification of the results in Section \ref{section:QSD}.
In Appendix \ref{AppA01}, we will prove some propositions whose proofs are postponed.

\subsection*{Acknowledgements}

K. Noba stayed at Centro de Investigaci\'{o}n en Matem\'{a}ticas in Mexico as a JSPS Overseas Research Fellow and received support regarding the research environment there. 
K. Noba is grateful for their support during his visit.
K. Noba was supported by JSPS KAKENHI grant no. JP21K13807. 
K. Yamato was supported by JSPS KAKENHI grant no. JP23KJ0236.
Both authors were supported by JSPS Open Partnership Joint Research Projects grant no. JPJSBP120209921.

\section{Preliminary} \label{section:preliminary}

{In the following, we fix a standard Markov process $X=(\Omega, \cF, \cF_t, X_t, \theta_t, \bP_x)$ whose state space is an interval $I \subset \bR$ with cemetery point $\Delta$
(for the definition, see, e.g., \cite[Definition 1.9.2]{BluGet1968}).
We set $I_{\Delta} := \Delta \cup \{ \Delta \}$. 
We always assume that $X$ has no negative jumps:
\begin{align}
	\bP_{z}[ \tau_{x} < \tau_{y}] = 0 \quad (\text{$x, y, z \in I$ with $x < y < z$}), \label{}
\end{align}
where $\tau_{A} := \inf \{ t > 0 \mid X_{t} \in A \}$ denotes the first hitting time of the set $A$ and we especially write $\tau_{x} := \tau_{\{x\}} \ (x \in I)$.
Set $\ell_{1} := \inf I$ and $\ell_{2} := \sup I$.
For $x \in I$, we write 
\begin{align}
I_{\geq x}:= I\cap [x, \infty),& \quad 
I_{>x}:= I\cap (x, \infty), \\
I_{\leq x }:=I\cap(-\infty, x],& \quad
I_{<x }:=I\cap(-\infty, x). 
\end{align}
}

\subsection{Local times, potential densities and excursion measures} \label{section:localTimeExcursionMeasure}

In this section, we recall some properties of local times and excursion measures of standard processes.
\par
%Let %$(\{X_{t}\}_{t \geq 0},\{\bP_{x}\}_{x \in I})$ 
%{$X=(\Omega, \cF, \cF_t, X_t, \theta_t, \bP_x)$}
%be a standard Markov process {(for the definition, see, e.g., \cite[Definition 1.9.2]{BluGet1968})} {with state space} an interval $I {\subset \bR}$ {with cemetery point $\Delta$}. 
%{For simplicity, we write $I_\Delta=\Delta\cup\{\Delta\}$.} 
%{Here, } 
%{We assume that $X$ has no negative jumps}:
%\begin{align}
%	\bP_{z}[ \tau_{y} < \tau_{x}] = 1 \quad (\text{$x, y, z \in I$ with $x < y < z$}), \label{}
%\end{align}
%where $\tau_{A} := \inf \{ t > 0 \mid X_{t} \in A \}$ denotes \textit{the first hitting time of the set $A$} and we especially write $\tau_{x} := \tau_{\{x\}} \ (x \in I)$.
%{[I think irreducibility is included in (A1). $\bP_{x}[\tau_{I \setminus \{x\} } = 0] =  1$ follows from \cite[Remark 2]{NobaGeneralizedScaleFunc}. Precisely from the result of Theorem 1 (vi) in Salisbury(1986), only upward jumps can occur after the holding time. In other words, in this case, it may be only stagnant at the lower end of $I$. $\bP_x [\tau_0 =0]=1$ is necessary?]}
%{Set $\ell_{1} := \inf I$ and $\ell_{2} := \sup I$}.
For $q \geq 0$, let us denote the $q$-resolvent of $X$ by $R^{(q)}$, i.e.,
\begin{align}
	R^{(q)}f(x) := \bE_{x}\left[ \int_{0}^{\infty}\mathrm{e}^{-qt}f(X_{t})dt \right] =  \int_{0}^{\infty}\mathrm{e}^{-qt}\bE_{x}[f(X_{t})]dt \label{}
\end{align}
for every non-negative measurable function $f$.
{In the present paper,} we {always} suppose the following conditions:
\begin{enumerate}
	\item[(A1)] The map $(x,y) \longmapsto \bE_{x}[\mathrm{e}^{-\tau_{y}}]$ is $\cB(I) \otimes \cB(I)$-measurable.
	\item[(A2)] For every $x, y \in {I}$ with $x<y$, it holds $\bP_{y}[\tau_{x} < \infty] > 0$. 
	\item[(A3)] The process $X$ has \textit{a reference measure} $m$ on $I$,
	{that is, there exists a Radon measure $m$ on $I$ and it holds for any measurable set $A \subset I$ that
	\begin{align}
		R^{(1)}1_{A}(x) = 0 \ \text{for every $x \in I$ if and only if } m(A) = 0. \label{}
	\end{align}
	(For the detail, see, e.g., \cite[p.196]{BluGet1968}.)}
	%{Here, the definition of reference measure follows \cite[pp.196]{BluGet1968}.}
	%Precisely, there is a Radon measure $m$ on $I$ and for $q \geq 0$,
	%$R^{(q)}f(x)$ is absolutely continuous w.r.t.\ $m$ for every non-negative measurable function $f$.
\end{enumerate}

%{According to \cite{GemanHorowitz}{[I guess we should add which pages have the information. Can you write down what page it is on?]}, we classify the states in $I$.
%The state $x \in I$ is called \textit{regular} when $\bP_{x}[\tau_{x} = 0] = 1$, \textit{polar} when $\bE_{y}[\mathrm{e}^{-\tau_{x}}] = 0$ for every $y \in I$, and \textit{irregular} otherwise.
%}
%{[I confirmed \cite{GemanHorowitz}. In this paper, the authors seemed to admit the existence of points irregular and polar in p.28...?]}
According to $(11.1)$ in \cite[Chapter I]{BluGet1968} we classify the states in $I$.
The state $x \in I$ is called \textit{regular} for itself when $\bP_{x}[\tau_{x} = 0] = 1$ and \textit{irregular} for itself otherwise. Henceforth, we often omit ``itself" and write only ``regular" and ``irregular". 
According to $(3.1)$ in \cite[Chapter II]{BluGet1968} the stare $x\in I$ is called \textit{polar} when $\bP_{y}[\tau_{x} {<\infty}] = 0$ %{[please comfirm.]} 
for all $y\in I$. 
Note that (A2) ensures that every $x \in I \setminus \{\ell_{2}\}$ is regular or irregular while the upper end point $\ell_{2}$ may be polar when $\ell_{2} \in I$.

\begin{Prop}\label{Prop101}
	The measure $m$ satisfies $m((x, y))>0$ for $x, y \in I$ with $x< y$. 
\end{Prop}
The proof of Proposition \ref{Prop101} is written in Appendix \ref{AppA01}.

By \cite[Theorem 18.4]{GemanHorowitz}, there exists a family of processes $\{ L^{x} \}_{x \in I}$ with $L^{x} = \{ L^{x}_{t} \}_{t \geq 0}$ which we call the \textit{local times} of $X$ such that the following formulas hold:
\begin{align}
	\int_{0}^{t}f(X_{s})ds &= \int_{I}f(y)L^{y}_{t}m(dy) \quad \bP_x\text{-a.s}. \label{occupationTimeFormula} 
\end{align}
for every $t\geq 0$, non-negative measurable function $f$ and $x \in I$.
We remark that while in the occupation density formula given in \cite[Theorem 18.4]{GemanHorowitz} there is a term related to the occupation time on polar points, it can be neglected because in our situation the only possible polar point is a single point, $\ell_{2}$, and it has no effect. 
Here, the process $L^x$ is a continuous additive functional when $x$ is regular, and is represented as 
\begin{align}
L^{x}_{t} = c^{x} \sharp\{0 \leq s <t: X_s = x\} \quad (t\geq 0)
\end{align} 
for some constant $c^{x}> 0$ when $x$ is irregular. 
Note that under the definition of the reference measure of \cite[pp.196]{BluGet1968}, 
we can take the local times above to satisfy $\bP_{y} (L^{x}\equiv 0)<1$ for some $y \in I$.
%{[Why is this necessary? This does not seem true when $x$ is {irregular}. Indeed, when the lifetime $\zeta$ is given by an independent exponential time, it can happen that $\bP_{y}[\zeta < \tau_{x}] > 0$ for every $y$.]}
%{[If $L^{x}\equiv 0$, we cannot define an excursion measure at $x$. I don't understand the second half of the comment.]}
Using \eqref{occupationTimeFormula}, we get
\begin{align}
	\bE_x \sbra{\int_0^Te^{-qt} f(X_t) dt} = \int_{{I}}f(y)\bE_{x}\left[ \int_{(0, T]}\mathrm{e}^{-qt} dL^{y}_{t} \right] m(dy) \label{potentialDensityFormula}
\end{align}
for $q \geq 0$ and non-negative random variable $T \in \cF$.

%{
%	[Since the set $\{ t \geq 0 \mid X_{t} = x \}$ is discrete $\bP_{y}$-a.s. for every $y \in I$,
%	it follows from \eqref{occupationTimeFormula} that $L_{t}^{x} m(\{x\}) = 0$ $P_{y}$-a.s.
%	Is it better to set $m\{x\} = 0$ for an irregular $x$?
%	]
%}
%{[I guess we do not need to add this assumption to $m$.]}
%\begin{Rem} \label{rem:assumptionOnReferenceMeas}
%	{[Added.]}
%	Let $I_{r}$ be the set of regular points of $X$; $I_{r} := \{ y \in I \mid \bP_{y}[\tau_{y} = 0] = 1 \}$.
%	For $y \in I_{r}^{c}$, the set $\{ t > 0 \mid X_{s} = y \}$ is $\bP_{x}$-a.s.\ discrete for every $x \in I$ and thus $\bE_{x}[\int_{0}^{\infty}1_{\{y\}}(X_{s})ds] = 0$.
%	Let $I_{m}$ be the set of atoms of $m$; $I_{m} := \{ y \in I \mid m\{ y \} > 0 \}$.
%	Since the set $I_{m}$ is at most countable, when we define
%	\begin{align}
%		\tilde{m}(dx) := 1 \{ x \in I_{r} \cup I_{m}^{c} \} m(dx), \label{}
%	\end{align}
%	we easily see the condition (A3) still holds for $\tilde{m}$.
%	Thus, we may assume without loss of generality that $m\{y\} > 0$ implies $y$ is regular.
%	Henceforth, we always assume that holds.
%\end{Rem}

For $x \in I$ {which is regular,} %{[I think omitting "for itself" does not cause any confusion, and it seems desirable because "polar" is weird. When we need to consider the regularity for a set $A$, we can just write the definition explicitly (in fact, we already have done so).]}
let $\eta^{x}$ be the inverse local time of $x$, the right-continuous inverse of $t \mapsto L^{x}_{t}$, i.e., $\eta^{x}_{t} := \inf \{ s > 0 \mid L^{x}_{s} > t \}$.
Let $n_{x}%(de) \ (e \in \bD)
$ %{[I feel $n_{x}(de) \ (e \in \bD)$ is better]} 
be the excursion measure {on $\bD_x$} away from $x$ associated with $L^{x}_{t}$, where $\bD_{x}$ denotes the space of c\'adl\'ag paths from $[0,\infty)$ to $I$ which started from and stopped at $x$ (for background on general excursion theory, see \cite[Section IV]{BertoinLevy} or \cite{Itoexcursion}). %{[Do we need to take $\bD_{x}$, the space of c\'adl\'ag paths started from and stopped at $x$, instead of $\bD$?]}. 
%{[Either. I don't think it's necessary. ]}

Then from the excursion theory (see, e.g., \cite[p.121]{BertoinLevy}), it holds
\begin{align}
	-\log \bE_{x}[\mathrm{e}^{-q\eta^{x}_{1}}] =\delta_{x} q + n_{x}[1 - \mathrm{e}^{-q\tau_{x}}]\qquad (q\geq 0),  \label{107}
\end{align}
{for some $\delta_{x}\geq 0$}.
For $x \in I$ which is {irregular}, 
we define $n_{x}$ as {$1 / c_x$} times the law of $X$ started from $x$ and stopped at $x$. 
Then, \eqref{107} holds.
%{[When $x$ is irregular, we need to take $n_{x} := (1 / c^{x}) \bP_{x}[X_{\cdot \wedge \tau_{x}} \in de]$.]}
%{When $\ell_{2} \in I$ and it is polar, we define $n_{\ell_{2}}$ as the law of $X$ starting from $\ell_{2}$.}
%{[I guess we can define $n_{\ell_{2}}$ in the same way as other irregular cases.]}
 
\subsection{Scale functions and potential densities on intervals} \label{section:scaleFunc}

We recall the scale functions for standard processes with no negative jumps from \cite{NobaGeneralizedScaleFunc}. 
We also describe some simple properties of the scale functions that we will use in the later sections.

For $q \geq 0$, define the \textit{$q$-scale function} $W^{(q)}: I \times I \to [0,\infty)$ by
\begin{align}
	W^{(q)}(x,y) := \frac{1}{n_{y}[\mathrm{e}^{-q\tau_{x}^{-}}, \tau_{x}^{-} < \infty]} \quad %\text{if } 
	(x \leq y),
	\label{scaleFunc}
\end{align}
where we consider $1/ \infty = 0$ and $\tau^{-}_{x} =\tau_{I_{\leq x} } $,
and
\begin{align}
	W^{(q)}(x,y) := 0 \quad  \text{otherwise.} \label{scaleFunc-offDiagonal}
\end{align}
We especially write $W := W^{(0)}$.
\begin{Rem} 
Since $X$ has no negative jumps, it holds, for $x, y \in I$ with $x<y$, $\bP_y (\tau_x\neq\tau^-_x)=n_y (\tau_x\neq\tau^-_x)=0$. 
\end{Rem}

\begin{Rem} \label{rem:downwardRegularity}
	From \cite[Remark 3.2]{NobaGeneralizedScaleFunc}, every point $x \in {I_{> \ell_{1}}}$ is regular for $I_{<x}$, that is, $\bP_{x}[\tau_{I_{<x}} = 0] = 1$.
	We see from this that for $x,y \in I$ with $\ell_{1} < x \leq y$
	\begin{align}
		\tau_{x}^{-} = \tau_{I_{<x}} \quad n_{y}\text{-a.e.} \label{}
	\end{align}
	and $\tau_{I_{<x}} = 0$ $n_{x}$-a.e.
	{This implies $W^{(q)}(x,x) = W(x,x)$ for every $q > 0$.}
	%{[I'm not sure whether we should prove this explicitly.]}
	%{[I think we should write in detail if the referee asks.]}
	% We note that when $x$ is {irregular}, it holds that $\tau_{x} > \tau_{x}^{-}$ $n_{x}$-a.e.
	% For example, if $X$ is a spectrally positive compound Poisson process with negative drift on $\bR$,
	% we see for every $x \in \bR$ that $\bP_{x}[\tau_{x} = 0] = 0$ and $\bP_{x}[\tau_{x}^{-} = 0] = \bP_{x}[\tau_{(-\infty,x)} = 0] = 1$.
\end{Rem}
{The following propositions give the positivity and continuity of the scale functions. 
The proof of Proposition \ref{Prop202} and \ref{prop:continuityOfW} is written in Appendix \ref{AppA01}.
}
\begin{Prop}\label{Prop202}%[{{[Added.]}}]
	For $q \geq 0$ and $x,y \in I$ with $x < y$,
	\begin{align}
		W^{(q)}(x,y) \in (0,\infty). \label{}
	\end{align}
\end{Prop}

%{
%\begin{Prop}
%For $q\geq 0$ and $x \in I$ with $x > \ell_1$, we have 
%\begin{align}
%W^{(q)} (x, x) m(\{x\}) =0.  \label{401}
%\end{align}
%\end{Prop}
%}
%\begin{proof}
%{
%We fix $x\in I$ and $q\geq 0 $. 
%}
%\par
%{
%Suppose that the $R^{(0)}1_{\{x\}} ( y)$ is equal to $0$ for all $y \in I$. Then $m(\{x\})=0$ holds by the definition of $m$ and \eqref{401} is obvious.  
%}
%\par
%{
%Suppose that the $R^{(0)}1_{\{x\}} ( y)$ is positive for some $y \in I$. 
%Then, by the definition of the reference measure and \eqref{potentialDensityFormula}, it holds $m(\{x\})>0$. 
%If $x$ is regular, then the excursion measure $n_x$ is an infinite measure. Thus $n_x\sbra{e^{-q\tau^-_x}, \tau^-_x<\infty }=\infty$ and $W^{(q)} (x,x)=0$ hold. 
%If $x$ is irregular, then $R^{(0)}1_{\{x\}} ( y)>0$ implies that $x$ is a holding point. 
%This fact contradicts Remark \ref{rem:downwardRegularity}. %\cite[Remark 2]{NobaGeneralizedScaleFunc}.
%}
%\end{proof}

\begin{Rem} \label{rem:W(x,x)m(x)=0}
	%{[Added.]}
	From \cite[Remark 2.1]{NobaGeneralizedScaleFunc}, the point $y \in I \setminus \{ \ell_{1} \}$ cannot be a holding point: $\bP_{y}[\tau_{I \setminus \{y\}} > 0] = 0$.
	Thus, if $y \in I$ is regular, the excursion measure $n_{y}$ is an infinite measure and $W^{(q)}(x,x) = 0$ for every $q \geq 0$.
	On the other hand, if $y$ is irregular, then the potential measure $m$ does not have a mass at $y$, and it holds $m(\{y\})=0$ by the definition of $m$.
	%Combining this with Remark \ref{rem:assumptionOnReferenceMeas}, 
	Therefore, the following equality holds:
	\begin{align}
		W^{(q)}(x,x)m\{x\} = 0 \quad (q \geq 0,\
		 x \in I \setminus \{ \ell_{1}\}). \label{}
	\end{align}
\end{Rem}
\begin{Prop} \label{prop:continuityOfW}
	For every $q \geq 0$ and $y\in I$, the function $I_{<y} \ni x \longmapsto W^{(q)}(x,y)$ is continuous.
\end{Prop}

\begin{Rem}
	The continuity of the function $y \mapsto W^{(q)}(x, y)$ depends on how we take the measure $m$ and the processes $\{ L^{x} \}_{x > 0}$, so we cannot be certain here. 
\end{Rem}

One of the importance of the scale functions is that they allow us to concretely represent the Laplace transform of the exit times and the potential density killed on exiting an interval.
\begin{Prop}[{\cite[Theorem 3.4]{NobaGeneralizedScaleFunc}}] \label{prop:exitProblem1}
	Let $q \geq 0$.
	For $x,y,z \in I$ with $x < y \leq z$,
	it holds 
	\begin{align}
		\bE_{y}[\mathrm{e}^{-q \tau^-_{x}}, \tau^-_{x} < \tau_{z}^{+}]
		= \frac{W^{(q)}(y,z)}{W^{(q)}(x,z)}, \label{eq56}
	\end{align}
	where $\tau_{z}^{+} := \tau_{I_{\geq z}}$.
\end{Prop}

\begin{Thm}[{\cite[Theorem 3.6]{NobaGeneralizedScaleFunc}}] \label{thm:potentialDensity}
	Let $q \geq 0$ and $x,y,z,u \in I$ with $y, u \in (x, z)$.
	It holds
	\begin{align}
		\bE_{y}\left[ \int_{{(0,\tau^-_{x} \wedge \tau_{z}^{+}]}} \mathrm{e}^{-qt} dL^{u}_{t} \right]
		= &\frac{W^{(q)}(x,u)W^{(q)}(y,z)}{W^{(q)}(x,z)} - W^{(q)}(y,u) \label{} \\
		= &W^{(q)}(x,u) \bE_{y}[\mathrm{e}^{-q\tau^-_{x}}, \tau^-_{x} < \tau_{z}^{+}] - W^{(q)}(y,u). \label{potentialDensity}
	\end{align}
\end{Thm}

By regarding "$b$" in the proofs of {\cite[Lemma 3.5 and Theorem 3.6]{NobaGeneralizedScaleFunc}} as "$-\infty$" and "$T_b^{-}$" as "$\infty$", we can prove the following theorem. Since the proof is consequently almost the same as that of Theorem \ref{thm:potentialDensity}, we omit the proof. 
\begin{Thm}\label{Thm201}
	Let $q \geq 0$ and $x,y,z \in I$ with $y , z\in I_{>x}$.
	It holds
	\begin{align}
		\bE_{y}\left[ \int_{(0, \tau^-_{x}] } \mathrm{e}^{-qt} dL^{z}_{t} \right]
		= W^{(q)}(x,z) \bE_{y}[\mathrm{e}^{-q\tau^-_{x}},\tau^-_x<\infty] - W^{(q)}(y,z). \label{a001}
	\end{align}
\end{Thm}

Note that by \eqref{potentialDensityFormula}, the function \eqref{potentialDensity} and \eqref{a001} are the potential densities of $X$ started from $y$ killed on exiting the interval $(x,z)$ and $I_{>x}$, respectively.

We also have the representation of the exit time from the upper end of an interval.

\begin{Prop}[{\cite[Corollary 3.7]{NobaGeneralizedScaleFunc}}] \label{prop:exitProblemZ}
	%{[Added.]}
	For $q \geq 0$ and $x,y \in I$, define
	\begin{align}
		Z^{(q)}(x,y) := 
		%\left\{
		%\begin{aligned}
			&1 + q\int_{(x,y)}W^{(q)}(x,u)m(du) .%& (x < y), \\
			%&1 & (x \geq y).
		%\end{aligned}
		%\right.
		\label{scaleFuncZ}
	\end{align}
	It holds for $x,y,z \in I$ with $y \in (x,z)$
	\begin{align}
		\bE_{y}[\mathrm{e}^{-q\tau_{z}^{+}}, \tau_{z}^{+} < \tau^-_{x}] = Z^{(q)}(y,z) - \frac{W^{(q)}(y,z)}{W^{(q)}(x,z)} Z^{(q)}(x,z). \label{eq54}
	\end{align}
\end{Prop}
Note that in this paper, $\int_{(x, y)}$ with $y < x$ denotes the integration on the empty set, and the value is equal to $0$. 
\begin{Rem}
In \cite{NobaGeneralizedScaleFunc}, a stronger condition than (A2) is assumed, and the statements of Proposition \ref{prop:exitProblem1} and \ref{prop:exitProblemZ} and Theorem \ref{thm:potentialDensity} are slightly generalized here.
However, we can obtain them by the same proof. 
\end{Rem}

\section{Analyticity of scale functions and the resolvent identity} \label{section:analyticExtension}

In this section, we prove two important properties of scale functions. 
The first is \eqref{eq78},
where, for the functions $f,g: I^{2} \to \bR$ and $x, y \in I$, we define
\begin{align}
	f \otimes g (x,y) := \int_{(x, y)}f(x,u)g(u,y) m(du) \label{eq09}
\end{align}
when $f$ and $g$ are non-negative or
\begin{align}
	\int_{(x, y)}|f(x,u)g(u,y)| m(du) < \infty. \label{}
\end{align}
The second is the series expansion \eqref{304} for $q\in\bC$,
where $f^{\otimes 1} := f$ and $f^{\otimes n} := f \otimes f^{n-1} \ (n \geq 2)$. 
We also show similar properties for $Z^{(q)}$.
Note that the product $\otimes$ is associative, that is, $f \otimes (g \otimes h) = (f \otimes g) \otimes h$, which follows from Fubini's theorem. 
\begin{Rem}
	The map $x \mapsto W^{\otimes n} (x, y)$ with $y\in I$ is non-increasing. 
	This fact can be confirmed inductively for $n$. 
\end{Rem}
{Before proving the main results in this section, we introduce the following proposition about finiteness, which is useful to compute some limits.

\begin{Prop} \label{prop:integrabilityOfW}
	For every $q \geq 0$ and $x, y, z \in I$ with $x < y < z $, it holds $\bE_{y}[\tau_{x} \wedge \tau_{z}^{+}] < \infty$ and
	\begin{align}
		\int_{(x, z]}W^{(q)}(x,u) m(du) < \infty. \label{eq24}
	\end{align}
	In addition, it holds
	\begin{align}
		\lim_{x \to z} \int_{(x,z]}W^{(q)}(x,u) m(du) = 0. \label{eq17}
	\end{align}
\end{Prop}

\begin{proof}
	Let $q > 0$.
	From Theorem \ref{Thm201}, we see for every $x, y \in I$ with $x < y$
	\begin{align}
		q^{-1} &\geq\int_{0}^{\infty}\mathrm{e}^{-qt} \bP_{y}[x < X_{t} \leq y, \tau_{x}  > t] dt \label{} \\
		&= \bE_{y}[\mathrm{e}^{-q\tau_{x}},\tau_x<\infty]\int_{(x, y]} W^{(q)}(x,u) m(du). \label{eq20-2} 
	\end{align}
	Since it obviously holds $W^{(q)}(x,z)m\{z\} < \infty$, we see \eqref{eq24} holds for $q > 0$.
	From the definition of $W^{(q)}$, it clearly holds $0 \leq W(x,u) \leq W^{(q)}(x,u)$,
	and we have \eqref{eq24} for $q = 0$.
	%{[Added.]}
	From Theorem \ref{thm:potentialDensity}, it holds
	\begin{align}
		\bE_{y}[\tau_{x} \wedge \tau_{z}^{+}] &= \int_{0}^{\infty} \bP_{y}[\tau_{x} \wedge \tau_{z}^{+} > t] dt \label{} \\
		&=\bP_{y}[\tau_{x} < \tau_{z}^{+}] \int_{(x,z)}W(x,u)m(du) - \int_{(y,z)}W(y,u)m(du) < \infty, \label{} 
	\end{align}
	and we see $\bE_{y}[\tau_{x} \wedge \tau_{z}^{+}] < \infty$.
	From Proposition \ref{prop:continuityOfW}, Remark \ref{rem:W(x,x)m(x)=0} and the dominated convergence theorem we have
	\begin{align}
		\lim_{x \to z} \int_{(x,z]} W^{(q)}(x,u) m(du) = W^{(q)}(z,z)m\{z\} = 0, \label{}
	\end{align}
	{and the proof is complete.}
\end{proof}

%For the functions $f,g: {{I}}^{2} \to \bR$, we define the product 
%\begin{align}
%f \otimes g (x,y) := \int_{(x, y)}f(x,u)g(u,y) m(du). \label{}
%\end{align}
%when $f$ and $g$ are non-negative or
%\begin{align}
%	\int_{(x, y)}|f(x,u)g(u,y)| m(du) < \infty. \label{}
%\end{align}
%{[Wrote the definition explicitly.]}

To prove \eqref{304}, we first show that the RHS of \eqref{304} is well-defined and analytic on $\bC$ and is the unique solution of a Volterra integral equation.

\begin{Prop} \label{prop:defOfM}
	For every $n \geq 1$ and $x ,y \in I$, it holds
	\begin{align}
		W^{\otimes n} (x,y) \leq \frac{W(x,y)}{(n-1)!} \left( \int_{(x, y)}W(x,u) m(du) \right)^{n-1}, \label{eq01}
	\end{align}
	which implies for $q \in \bC$ it holds
	\begin{align}
		\sum_{n \geq 0}|q|^{n} W^{\otimes (n+1)}(x,y) \leq W(x,y) \mathrm{e}^{|q|\int_{(x, y)}W(x,u) m(du)} \label{eq03}
	\end{align}
	and
	\begin{align}
		M^{(q)}(x,y) := \sum_{n \geq 0} q^{n}W^{\otimes (n+1)}(x,y) \label{seriesExpansionOfM}
	\end{align}
	defines an entire function in $q \in \bC$ for every fixed $x, y \in I$.
	In addition, for every $x, z \in I$ with $x < z  $ the function $f = M^{(q)}(x,\cdot)$ is the unique function satisfying for $y \in (x,z]$
	\begin{align}
		f(y) = W(x,y) + q \int_{(x, y)}f(u) W (u,y) m(du) \label{IntegralEq} 
	\end{align}
	and
	\begin{align}
		\int_{(x, y)}|f(u)|W(u,y) m(du) < \infty.  \label{eq04}
	\end{align}
\end{Prop}

\begin{proof}
	The statement is obvious when $x>y$ since $W^{\otimes k}(x, y)= 0$ for $k\in\bN$. 
	Therefore, we assume $x\leq y$ in the following. 
	We show \eqref{eq01} by induction.
	The case $n = 1$ is obvious.
	Define $F_{x}(y) := \int_{( x, y)}W(x,u) m(du)$ for $x \leq y$.
	Suppose \eqref{eq01} holds for some $n \geq 1$.
	Then it holds
	\begin{align}
		W^{\otimes (n+1)}(x,y) &= \int_{(x, y)}W^{\otimes n}(x,u)W(u,y) m(du) \label{} \\
		&\leq \frac{1}{(n-1)!}\int_{(x, y)}W(x,u)W(u,y) F_{x}(u)^{n-1} m(du) \label{}\\
		&\leq \frac{W(x,y)}{(n-1)!}\int_{(x, y)}W(x,u) F_{x}(u)^{n-1} m(du) \label{}\\ 
		&= \frac{W(x,y)}{(n-1)!}\int_{(x, y)} F_{x}(u)^{n-1} dF_{x}(u) \label{} \\
		&\leq \frac{W(x,y)}{n!}F_{x}(y)^{n}. \label{eq02}
	\end{align}
	Thus, we obtain \eqref{eq01}{, and \eqref{eq03} as well}.
	Next, we show $f = M^{(q)}(x, \cdot)$ is the unique function satisfying \eqref{IntegralEq} and \eqref{eq04} for every fixed $x, z \in I$ with $x < z$.
	The function $f= M^{(q)}(x,\cdot)$ satisfies \eqref{eq04} since for $q \neq 0$
	\begin{align}
		\int_{(x, y)}|f(u)|W(u,y){m(du)}=&\sum_{n\geq 0}{|q|}^n \int_{(x, y)} W^{\otimes(n+1)}(x, u)W(u ,y) m(du)
		\\
		=& \sum_{n\geq 0}{|q|}^nW^{\otimes (n+2)}(x, y)\\
		=& \frac{1}{|q|}\rbra{M^{(|q|)}(x, y)-W(x, y)} < \infty, \label{316}
	\end{align}
	and thus satisfies \eqref{IntegralEq}.
	Let $G$ be another solution and set $H := G - M^{(q)}(x,\cdot)$.
	Then it holds for $y \in (x,z]$
	\begin{align}
		H(y) = q \int_{(x,y)}H(u) W(u,y) m(du). \label{eq69}
	\end{align}
	Since it holds from \eqref{eq01} for every $n \geq 1$
	\begin{align}
		\int_{(x, y)}|H(u)|W^{\otimes n}(u,y) m(du) \leq \frac{F_{x}(y)^{n-1}}{(n-1)!} \int_{(x,y)} |H(u)| W(u,y) m(du) < \infty, \label{eq50}
	\end{align}
	we see inductively that
	\begin{align}
		H(y) = q^{n} \int_{(x, y)}H(u)W^{\otimes n}(u,y) m(du). \label{1}
	\end{align}
	Then from \eqref{eq50} we have
	\begin{align}
	|H(y)| \leq \frac{q^n F_{x}(y)^{n-1} }{(n-1)!}\int_{(x, y)}|H(u)|W(u,y) m(du) \xrightarrow[]{n \to \infty} 0,\label{430}
	\end{align}
	and we obtain the uniqueness.
\end{proof}

A review of the proof of Proposition \ref{prop:defOfM} shows that the uniqueness of the solution can be slightly extended.

\begin{Cor}\label{Cor404}
	For every $x, z \in I$ with $x < z$, the function $f = M^{(q)}(x,\cdot)$ in Proposition \ref{prop:defOfM} is a unique solution of \eqref{IntegralEq} in the following sense: for some $m$-null set $N$, if $f:(x,z] \to \bR$ satisfies \eqref{IntegralEq} and \eqref{eq04} for $y \in (x,z] \setminus N$, it holds $f(y) = M^{(q)}(x,y) $ for $ y \in (x,z] \setminus N$.
\end{Cor}
The proof of Corollary \ref{Cor404} can be obtained by changing the function $G$ in the proof so that it satisfies \eqref{IntegralEq} and \eqref{eq04} on $(x,z] \setminus N$ for some $m$-null set $N$. 
This is because \eqref{eq69}-\eqref{430} hold for $y \in (x,z] \setminus N$. 
%{We omit the detail of the proof.}
%{[I don't think this excuse is necessary because the above description gives everything for the proof.]}

From Proposition \ref{prop:defOfM}, in order to show \eqref{304} it suffices to show that $W^{(q)}$ is a solution of the equation \eqref{IntegralEq}.
We now prove \eqref{eq78} for $q,r \geq 0$, from which we see that $f = W^{(q)}$ solves \eqref{IntegralEq}.
% is the generalization of \eqref{304}, holds in particular cases. 

\begin{Thm} \label{thm:WequalM}
	For every $q \geq 0$, it holds 
	\begin{align}
		W^{(q)}(x, y) = M^{(q)}(x, y )  \quad (x, y \in I). \label{329}
	\end{align}
\end{Thm}

\begin{proof}
When $x>y$, the assertion is obvious since both sides of \eqref{329} are zero. 
The case with $x=y$ is also obvious since $M^{(q)}(x,x)=W(x, x)=W^{(q)}(x, x )$ by the definitions of $\otimes$ and {Remark \ref{rem:downwardRegularity}}.
Thus, we prove \eqref{329} with $x< y$ in the following. 
\par
	Let $q \geq 0$ and $x, z \in I$ with $x < z$.
	{We show $f = W^{(q)}(x,\cdot)$ satisfies the equation \eqref{IntegralEq} on $(x,z]$.}
	For every $t \geq 0$ and $y \in (x,z)$, it holds from Proposition \ref{prop:exitProblem1}
	\begin{align}
		\bE_{y}[W^{(q)}(X_{t},z), \tau^-_{x} \wedge \tau_{z}^{+} > t]
		&= W^{(q)}(x,z)\bE_{y}[\bE_{X_{t}}[\mathrm{e}^{-q\tau^-_{x}}, \tau^-_{x} < \tau_{z}^{+}] , \tau^-_{x} \wedge \tau_{z}^{+}> t] \label{} \\
		&= \mathrm{e}^{qt}W^{(q)}(x,z) \bE_{y}[\mathrm{e}^{-q\tau^-_{x}},t < \tau^-_{x} < \tau_{z}^{+}]. \label{}
	\end{align}
	Thus, we have for $r, q \geq 0$ with $r\neq q$,
	\begin{align}
		&\int_{0}^{\infty}\mathrm{e}^{-rt} \bE_{y}[W^{(q)}(X_{t},z), {\tau^-_{x}} \wedge \tau_{z}^{+}> t]dt \label{} \\
		= &W^{(q)}(x,z)\bE_{y} \left[ \mathrm{e}^{-q{\tau^-_{x}}} \int_{0}^{{\tau^-_{x}}} \mathrm{e}^{-(r-q)t}dt, {\tau^-_{x}} < \tau_{z}^{+}  \right] \label{} \\
		= &\frac{W^{(q)}(x,z)}{r-q}\bE_{y} \left[ \mathrm{e}^{-q{\tau^-_{x}}} - \mathrm{e}^{-r{\tau^-_{x}}} , {\tau^-_{x}} < \tau_{z}^{+}  \right] \label{} \\
		= & \frac{W^{(q)}(x,z)}{r-q} \left( \frac{W^{(q)}(y,z)}{W^{(q)}(x,z)} - \frac{W^{(r)}(y,z)}{W^{(r)}(x,z)} \right). \label{eq05}
	\end{align}
	On the other hand, from Theorem \ref{thm:potentialDensity}, we have
	\begin{align}
		&\int_{0}^{\infty}\mathrm{e}^{-rt} \bE_{y}[W^{(q)}(X_{t},z), {\tau^-_{x}} \wedge \tau_{z}^{+} > t]dt \label{} \\
		= & \int_{(x, z)}\left(\frac{W^{(r)}(x,u)W^{(r)}(y,z)}{W^{(r)}(x,z)} - W^{(r)}(y,u)\right) W^{(q)}(u,z) m(du) \label{} \\
		= & \frac{W^{(r)}(y,z)}{W^{(r)}(x,z)} W^{(r)}\otimes W^{(q)}(x,z) - W^{(r)} \otimes W^{(q)}(y,z). \label{eq06}
	\end{align}
	From \eqref{eq05} and \eqref{eq06}, we see
	\begin{align}
	%\begin{aligned}
		&\frac{W^{(q)}(x,z) + (r-q)W^{(r)}\otimes W^{(q)}(x,z)}{W^{(r)}(x,z)} \\
		& = \frac{W^{(q)}(y,z) + (r-q)W^{(r)}\otimes W^{(q)}(y,z)}{W^{(r)}(y,z)}. 
		%\end{aligned}
		\label{eq08}
	\end{align}
	Thus, the RHS of \eqref{eq08} is a constant in $y \in (x,z)$ and we write it $c(q,r)$. 
	We prove $c(q,r) = 1$ for every $q,r \geq 0$. 
	
	First, let $r \geq q$.
	It clearly holds $W^{(q)}(y,z) / W^{(r)}(y,z) \leq 1$ by the definition of $W^{(q)}$.
	Since it holds
	\begin{align}
		W^{(r)} \otimes W^{(q)}(y,z) \leq W^{(q)}(y,z)\int_{(y, z)}W^{(r)}(y,u) m(du), \label{}
	\end{align}
	we see from \eqref{eq17}
	\begin{align}
		\frac{W^{(r)} \otimes W^{(q)}(y,z)}{W^{(r)}(y,z)} \leq \int_{(y, z)}W^{(r)}(y,u)m(du) \xrightarrow[]{y \to z-} 0. \label{}
	\end{align}
	By considering the limit $y \to z-$ in the RHS of \eqref{eq08}, we see for $r \geq q$
	\begin{align}
		c(q,r) = \lim_{y \to z-} \frac{W^{(q)}(y,z)}{W^{(r)}(y,z)} \geq 0. \label{eq11}
	\end{align}
	By \eqref{scaleFunc}, we have  
	\begin{align}
	\frac{W^{(q)}(y,z)}{W^{(r)}(y,z)}=\frac{n_z[\mathrm{e}^{-r {\tau^-_y}} ,{\tau^-_y}<\infty]}{n_z[\mathrm{e}^{-q {\tau^-_y}} ,{\tau^-_y}<\infty] }, 
	\end{align}
	and for any $\varepsilon>0$, 
	\begin{align}
	\frac{n_z[\mathrm{e}^{-r {\tau^-_y}} ,{\tau^-_y}<\varepsilon]}{n_z[{\tau^-_y}<\varepsilon]+ n_z[\varepsilon\leq {\tau^-_y}<\infty] }
	&\leq
	\frac{n_z[\mathrm{e}^{-r {\tau^-_y}},{\tau^-_y}<\infty]}{n_z[\mathrm{e}^{-q {\tau^-_y}} ,{\tau^-_y}<\infty] }\\
	&\leq
	\frac{n_z[{\tau^-_y}<\varepsilon]+n_z[\varepsilon\leq {\tau^-_y}<\infty]}{n_z[\mathrm{e}^{-q {\tau^-_y}} ,{\tau^-_y}<\varepsilon] }. 
	\end{align}
	By the dominated convergence theorem, we have 
	\begin{align}
	&\lim_{y \to z-}\frac{n_z[\mathrm{e}^{-r {\tau^-_y}} ,{\tau^-_y}<\varepsilon]}{n_z[{\tau^-_y}<\varepsilon]+ n_z[\varepsilon\leq {\tau^-_y}<\infty] }
	\geq \lim_{y \to z-}\frac{\mathrm{e}^{-r\varepsilon }n_z[{\tau^-_y}<\varepsilon]}{n_z[{\tau^-_y}<\varepsilon]+ n_z[\varepsilon\leq {\tau^-_y}<\infty] }=\mathrm{e}^{-r \varepsilon}\\
	&\lim_{y \to z-}\frac{n_z[{\tau^-_y}<\varepsilon]+n_z[\varepsilon\leq {\tau^-_y}<\infty]}{n_z[\mathrm{e}^{-q {\tau^-_y}} ,{\tau^-_y}<\varepsilon] }\leq \lim_{y \to z-}\frac{n_z[{\tau^-_y}<\varepsilon]+n_z[\varepsilon\leq {\tau^-_y}<\infty]}{\mathrm{e}^{-q\varepsilon }n_z[{\tau^-_y}<\varepsilon] }=\mathrm{e}^{q \varepsilon}. 
	\end{align}
	Thus, for any $\varepsilon>0$ it holds $\mathrm{e}^{-r \varepsilon}  \leq c(q,r)\leq \mathrm{e}^{q\varepsilon}$, which implies $c(q,r)=1$.
	{Once we know $\lim_{y \to z-} W^{(q)}(y,z) / W^{(r)}(y,z) = 1 \ (q \leq r)$,
	we can also show $c(q,r) = 1 \ (q > r)$ by the same argument.}
	Hence, we obtain
	\begin{align}
	\begin{aligned}
		W^{(q)}(y,z) - W^{(r)}(y,z) = (q-r) W^{(q)} \otimes W^{(r)}&(y,z) \\ & (y \in (x,z], \ q,r \geq 0), 
		\end{aligned}
		\label{eq14}
	\end{align}
	{where the case $y = z$ follows from Remark \ref{rem:downwardRegularity}.}
	Considering the case of $r = 0$ in \eqref{eq14}, we can see that $W^{(q)}(x , \cdot)$ satisfies \eqref{IntegralEq} and the proof is complete by Proposition \ref{prop:defOfM}.
\end{proof}
From Proposition \ref{prop:defOfM} and Theorem \ref{thm:WequalM}, for every $x,y \in I$, we can analytically extend the function $[0,\infty) \ni q \mapsto W^{(q)}(x,y)$ to the entire function by \eqref{304}.
In the following, we always regard $W^{(q)}(x,y)$ as such an entire function. 

To show \eqref{eq78} for $q,r \in \bC$, we first check that the RHS of \eqref{eq78} is well-defined for $q, r\in \bC$ and is analytic for $q$ and $r$.

\begin{Prop} \label{prop:commutativity}
	For $q,r \in \bC$, the product $W^{(q)} \otimes W^{(r)}$ is well-defined and commutative:
	\begin{align}
		W^{(q)} \otimes W^{(r)} (x, y) = W^{(r)} \otimes W^{(q)}(x, y), \quad x,y \in I. \label{commutativity}
	\end{align}
	In addition, for $x, y \in I$, the functions $\bC \ni q \mapsto W^{(q)} \otimes W^{(r)}(x,y)$ and $\bC \ni r \mapsto W^{(q)} \otimes W^{(r)}(x,y)$ are analytic on $\bC$.
\end{Prop}

\begin{proof}
	Since it holds
	\begin{align}
	&\int_{(x, y)} |W^{(q)}(x, u)W^{(r)}(u ,y)|m(du) \label{}\\
		\leq &\sum_{n \geq 0} |q|^{n}\int_{(x,y)} W^{\otimes(n+1)}(x,u) |W^{(r)}(u,y)| m(du) \label{} \\
		\leq & \mathrm{e}^{|q|\int_{(x,y)}W(x,v) m(dv)} \int_{(x,y)} W(x,u)W^{(|r|)}(u,y) m(du) \label{} \\
		\leq & \mathrm{e}^{|q|\int_{(x,y)}W(x,v) m(dv)} W^{(|r|)}(x,y) \int_{(x,y)} W(x,u) m(du)<\infty, \label{} 
	\end{align}
	where in the first inequality we used %{\eqref{329} and} 
	\eqref{eq03} and in the last inequality we used \eqref{eq24},	
	the product $W^{(q)} \otimes W^{(r)}$ is well-defined and it holds
	\begin{align}
		W^{(q)} \otimes W^{(r)}(x,y) = \sum_{n \geq 0} q^{n}\int_{(x,y)} W^{\otimes(n+1)}(x,u) W^{(r)}(u,y) m(du), \label{eq12} 
	\end{align}
	which implies the function $\bC \ni q \mapsto W^{(q)} \otimes W^{(r)}{(x,y)}$ is an entire function.
	From \eqref{eq12} and the associativity of the product $\otimes$, in order to show \eqref{commutativity} it is enough to show $W \otimes W^{(r)} = W^{(r)} \otimes W$, and that is clear from \eqref{eq12} for $r = 0$.
\end{proof}

\begin{Rem} \label{rem:Zexpansion}
	%{[Added.]}
	The function $[0,\infty) \ni q \mapsto Z^{(q)}(x,y) \ (x,y \in I)$ defined in \eqref{scaleFuncZ} is also extended to an entire function.
	Indeed, from \eqref{eq03}, it holds
	\begin{align}
		\sum_{n \geq 0} |q|^{n}\int_{(x,y)}W^{\otimes (n+1)}(x,u)m(du) \leq \mathrm{e}^{|q|\int_{(x,y)}W(x,u)m(du)} \int_{(x,y)}W(x,u)m(du) < \infty. \label{}
	\end{align}
	{
	Thus, it holds
	\begin{align}
		Z^{(q)}(x,y) = 1 + \sum_{n \geq 0} q^{n+1}\int_{(x,y)}W^{\otimes(n+1)}(x,u)m(du) \quad (x,y \in I, \ q \geq 0). \label{}
	\end{align}
	As in $W^{(q)}$, hereafter we always consider $Z^{(q)}(x,y)$ to be an entire function of $q$.
	}
\end{Rem}

The scale functions satisfy the resolvent identity.

\begin{Prop} \label{prop:resolventIdentity}
	For $q,r \in \bC$ and $x, y \in I$, the equality \eqref{eq78} holds. 
\end{Prop}

\begin{proof}
	We have already shown the case $q,r \geq 0$ in \eqref{eq14}.
	Since $W^{(q)} \otimes W^{(r)}(x,y)$ is the entire function in $q \in \bC$ from Proposition \ref{prop:commutativity},
	we easily see the desired result by applying the identity theorem twice.
\end{proof}
Now we can show a limit behavior on the ratio of scale functions in a little stronger form than \eqref{eq11}.
\begin{Cor} \label{cor:ratioLimitScaleFunc}
	For $q,r \in \bC$ and $u \in I$, it holds
	\begin{align}
		\lim_{\substack{|x - y| \to 0, \\ x \leq u \leq y,\\ x < y}} \frac{W^{(q)}(x,y)}{W^{(r)}(x,y)} = 1. \label{}
	\end{align}
	In particular, for every fixed $q \in \bR$, it holds $W^{(q)}(x,y) > 0$ for $x,y \in I$ such that $x < y$ and $y-x$ is sufficiently small.
\end{Cor}
\begin{proof}
	It is enough to show the case $r = 0$.
	Take $u \in I$. 
	When $u = \ell_{1}$, set $x_{0} := \ell_{1}$ and when $u > \ell_{1}$, take $x_{0} \in (\ell_{1},u)$ arbitrarily.
	From Theorem \ref{thm:WequalM} and \eqref{eq03}, it holds for $x \geq x_{0}$
	\begin{align}
		\left| \frac{W^{(q)}(x,y)}{W(x,y)} - 1 \right|
		=&\absol{ \sum_{n \geq 1} q^{n}\frac{W^{\otimes (n+1)}(x,y)}{W(x, y)}}
		\leq \sum_{n \geq 1} {|q|}^{n}\frac{W^{\otimes (n+1)}(x,y)}{W(x, y)} \\
		=&\frac{1}{W(x, y)} \sum_{n \geq 0} {|q|}^{n}{W^{\otimes (n+1)}(x,y)} -1\\
			\leq&{\mathrm{e}^{|q|\int_{(x,y)}W(x_{0},v)m(dv)} -1}. \label{}
	\end{align}
	Thus, we have the desired result by the dominated convergence theorem.
\end{proof}

A similar characterization of $Z^{(q)}$ by a Volterra integral equation and the series expansion holds.
\begin{Thm} \label{thm:characterizationOfZ}
	For $q \in \bC$ and $x, y \in I$, it holds
	\begin{align}
		\sum_{n \geq 2}|q|^{n} W^{\otimes (n-1)} \otimes \bar{W}(x,y) < \infty  \label{}
	\end{align}
	and
	\begin{align}
		Z^{(q)}(x,y) = 1 + q\bar{W}(x,y) + \sum_{n \geq 2}q^{n} W^{\otimes (n-1)} \otimes \bar{W}(x,y),  \label{eq59}
	\end{align}
	where
	\begin{align}
		\bar{W}(x,y) := \int_{(x,y)}W(x,u)m(du). \label{}
	\end{align}
	%{[We should move it to above.]}
	In addition, it holds
	\begin{align}
		|Z^{(q)}(x,y)| \leq 1 + |q|\bar{W}(x,y) \mathrm{e}^{|q|\bar{W}(x,y)}. \label{eq81}
	\end{align}
	Moreover, for $x,z \in I$ with $x < z$ the function $f = Z^{(q)}(\cdot,z)$ is characterized as a unique solution of a Volterra integral equation in the following sense: for a function $f:(x,z) \to \bR$ and a $m$-null set $N$, if it holds for every $y \in [x,z) \setminus N$
	\begin{align}
		f(y) = 1 + q \int_{(y,z)}W(y,u)f(u) m(du) \label{eq60}
	\end{align}
	and 
	\begin{align}
		\int_{(y,z)} W(y,u)|f(u)| m(dy) < \infty, \label{}
	\end{align}
	then $f(y) = Z^{(q)}(y,z) \ (y \in [x,z) \setminus N)$.
\end{Thm}

\begin{proof}
	The inequality \eqref{eq81} has already shown in Remark \ref{rem:Zexpansion}.
	Note that for $x,y \in I$ with $x < y$ and $n \geq 1$ it holds form Fubini's theorem
	\begin{align}
		W^{\otimes n} \otimes \bar{W}(x,y) = \int_{(x,y)}W^{\otimes (n+1)}(x,u)m(du). \label{eq61}
 	\end{align}
	{From Remark \ref{rem:Zexpansion}, we have \eqref{eq59}.}
	By the same argument in Proposition \ref{prop:defOfM} and Corollary \ref{Cor404}, we easily see the desired uniqueness.
\end{proof}

In the following corollary, we give an identity for $Z^{(q)}$ which corresponds to \eqref{eq78}.}
\begin{Cor} \label{cor:resolventIdentityForZ}
	For $q,r \in \bC$ and $x,y \in I$ with $x < y$ it holds
		\begin{align}
		Z^{(q)}(x,y) - Z^{(r)}(x,y) = (q-r) W^{(q)} \otimes Z^{(r)}(x,y)=(q-r) W^{(r)} \otimes Z^{(q)}(x,y). \label{eq62}
	\end{align}
\end{Cor}

\begin{proof}
	When $q = 0$, the equality \eqref{eq62} follows from Theorem \ref{thm:characterizationOfZ} and
	the case $r=0$ is obvious from the definition of $Z^{(r)}$.
	Let $r,q \neq 0$.
	By integrating both sides of \eqref{eq78} in $y$, we have for $x < y$
	\begin{align}
		&\frac{Z^{(q)}(x,y)-1}{q} - \frac{Z^{(r)}(x,y)-1}{r} \label{} \\
		=&\int_{(x, y)} \rbra{W^{(q)}(x,u) - W^{(r)}(x,u)} m(du) \label{}\\
		= &(q-r)\int_{(x,y)}m(du)\int_{(x,u)}W^{(q)}(x,v)W^{(r)}(v,u)m(dv) \label{} \\
		=& (q-r)\int_{(x,y)}W^{(q)}(x,v)m(dv)\int_{(v,y)}W^{(r)}(v,u)m(du) \label{} \\
		= & \frac{q-r}{r}W^{(q)}\otimes Z^{(r)}(x,y) - \frac{q-r}{qr}(Z^{(q)}(x,y) - 1), \label{}
	\end{align}
	{where in the third equality changing the order of integration is justified since
	\begin{align}
		&\int_{(x,y)}m(du)\int_{(x,u)}|W^{(q)}(x,v)W^{(r)}(v,u)|m(dv) \label{} \\
		\leq &\left(\int_{(x,y)}W^{(|r|)}(x,u)m(du)\right) \left(\int_{(x,y)}W^{(|q|)}(x,v)m(dv)\right) < \infty \label{}
	\end{align}
	by Proposition \ref{prop:integrabilityOfW}.
	}
	%{where in the second equality, we used the Fubini's theorem with the fact that 
	%\begin{align}
	%&\int_{(x,y)}m(du)\int_{(x,u)}\absol{W^{(q)}(x,v)W^{(r)}(v,u)}m(dv)\label{}\\
	%&\leq \int_{(x,y)}m(du)\int_{(x,u)}W^{(q^\prime)}(x,v)W^{(r^\prime)}(v,u)m(dv)\label{}\\
	%&=\frac{1}{(q^\prime- r^\prime)}\rbra{\frac{Z^{(q^\prime)}(x,y)-1}{q^\prime} - \frac{Z^{(r^\prime)}(x,y)-1}{r^\prime}}
	%<\infty, \label{}
	%\end{align}
	%for $q^\prime, r^\prime >0$ with $q^\prime>r^\prime$, $q^\prime >|q|$ and $r^\prime >|r|$. 
	%}
	Rearranging gives the first equality of \eqref{eq62}. 
	By exchanging $q$ and $r$ in the first equality of \eqref{eq62}, we see that the first term and the last term in \eqref{eq62} are equal.
\end{proof}

\section{Some results on the positivity of the scale functions} \label{section:positivity}
%{[I guess the following fact should not put in the section of scale functions.]}
{
	In this section, we investigate positivity of the scale function $W^{(q)}(x,y)$ for negative $q$.
}
Define for $x \in I_{<\ell_{2}}$ and $y \in I$ with $x < y$
{
\begin{align}
	\lambda^{[x,y]}_{0} = \sup \{ \lambda \geq 0 \mid \bE_{u}[ \mathrm{e}^{\lambda \tau_{x}^{-}}, \tau_{x}^{-} < \tau_{y}^{+}] < \infty \ \text{for every $u \in  (x,y)$} \} \in [0,\infty]. \label{}
\end{align}
}
{
It is clear that the function $I_{>x} \ni y \mapsto \lambda^{{[x,y]}}_{0}$ is non-increasing.
}
%{[It looks that $\lambda^{[x,y]}_{0}$ does not depend on $y$. %And I guess we should mention that 
%\begin{align}
%	\lambda^{[x,y]}_{0} = \sup \{ \lambda \geq 0 \mid \bE_{u}[ \mathrm{e}^{\lambda \tau_{x}^{-}}, \tau_{x}^{-} < \tau_{y}^{+}] < \infty \ \text{for some $u \in  (x,y)$} \} \in [0,\infty] \label{}
%\end{align}
%]}

\begin{Prop} \label{prop:positivityOfW}
	%{[$\lambda_{0}^{[x]}$ is replaced by $\lambda^{[x,z]}_{0}$. The assertion (i) is added.]}
	Let $x,z \in I$ with $x < z$.
	Then the following holds:
	\begin{enumerate}
		\item The function $I_{<z} \ni u \mapsto \lambda_{0}^{[u,z]}$ is non-decreasing. 
		\item For $q \in (-\lambda_{0}^{[x,z]},\infty)$ and $u,y \in [x,z]$ with $u < y$, it holds $W^{(q)}(u,y) > 0$.
		\item For $u,y \in [x,z]$ with $u < y$, the function $[-\lambda_{0}^{[x,z]},\infty) \cap \bR \ni q \mapsto W^{(q)}(u,y)$ is strictly increasing. 
	\end{enumerate}
\end{Prop}

\begin{proof}
	(i)
	For $u, v, y \in I_{<z}$ with $u < v < y$, 
	For $\lambda \in \bR $ from the strong Markov property at $\tau^-_v$,
	\begin{align}
	\bE_y\sbra{e^{\lambda \tau^-_u}, \tau^-_u <\tau^+_z}=&
	\bE_y\sbra{e^{\lambda \tau^-_v} \bE_{X_{\tau^-_v}}\sbra{ e^{\lambda \tau^-_u}, \tau^-_u< \tau^+_z}, \tau^-_v <\tau^+_z}\\
	=&\bE_y\sbra{e^{\lambda \tau^-_v}, \tau^-_v <\tau^+_z} \bE_{v}\sbra{ e^{\lambda \tau^-_u}, \tau^-_u< \tau^+_z}
	\end{align}
	Thus, if $\lambda<\lambda_0^{[u, z]}$, then $\bE_y\sbra{e^{\lambda \tau^-_u}, \tau^-_u <\tau^+_z}\in(0, \infty)$ and $\bE_v\sbra{e^{\lambda \tau^-_u}, \tau^-_u <\tau^+_z}\in(0, \infty)$.
	Thus, we have $\bE_y\sbra{e^{\lambda \tau^-_v}, \tau^-_v <\tau^+_z}$ for $y\in (v, z)$ and $\lambda<\lambda_0^{[v, z]}$. 
	Therefore, it holds $\lambda_0^{[u, z]}\leq \lambda_0^{[v, z]}$. 

	(ii)
	The case $q\geq 0$ is clear from Proposition \ref{Prop202}.
	Suppose $\lambda^{[x,z]}_{0} > 0$ and $W^{(q')}(u,y) = 0$ for some $q' \in (-\lambda^{[x,z]}_{0}, 0)$ and $u,y \in [x,z]$ with $u < y$.
	Proposition \ref{prop:exitProblem1} gives
	\begin{align}
		W^{(q)}(u,y) \bE_{v}[\mathrm{e}^{-q \tau_{u}^{-}}, \tau_{u}^{-} < \tau_{y}^{+}]
		= W^{(q)}(v,y) \quad (q \geq 0, \ v \in (u,y)).\label{507}
	\end{align}
	This identity can be analytically extended to $q \in\bC$ with $\Re q > -\lambda^{[x,z]}_{0} \ (\geq -\lambda^{[u,y]}_{0})$, and it follows $W^{(q')}(v,y) = 0$ for every $v \in (u,y)$.
	It is, however, impossible since it holds from {Corollary \ref{cor:ratioLimitScaleFunc}} %Remark \ref{rem:c(q,r)} 
	that for a sufficiently small $\delta > 0$
	\begin{align}
		W^{(q')}(v,y) > \frac{W(v,y)}{2} > 0 \quad (v \in (y-\delta,y)). \label{}
	\end{align}

	(iii)
	From Proposition \ref{prop:resolventIdentity} and (ii), we have for $r > q > -\lambda^{[x,z]}_{0}$
	\begin{align}
		W^{(r)}(u,y) - W^{(q)}(u,y) = (r-q) \int_{(u,y)}W^{(q)}(u,v)W^{(r)}(v,y)m(dv) > 0. \label{}
	\end{align}
	From the continuity of the map $q \mapsto W^{(q)}(u,y)$,
	we see $W^{(r)}(u,y) > W^{(-\lambda^{[x,z]}_{0})}(u,y)$.
\end{proof}

\begin{Rem}\label{Rem502}
	Proposition \ref{prop:positivityOfW} and the continuity $q \mapsto W^{(q)}(u,y)$ implies
	$W^{(-\lambda^{[x,z]}_{0})}(u,y) \geq 0 \ (u,y \in [x,z])$.
\end{Rem}

\section{Existence of quasi-stationary distributions: Inaccessible case} \label{section:QSD}

%{\color{blue} [I changed the name of this section from "One-side exit case"]}

In this section, we assume $\ell_{1} \in I$.
Then the interval $I$ can be identified with $[0,\ell] \ (\ell \in (0,\infty))$ or $[0,\ell) \ (\ell \in (0,\infty])$ by an appropriate order-preserving homeomorphism.
Thus, we may assume without loss of generality that $I$ is given by that form.
Our purpose in this section is to investigate the existence of a QSD of $X$ killed at $0$, that is, a probability distribution $\nu$ on $I$ such that
\begin{align}
	\bP_{\nu}[ X_{t} \in dy \mid \tau_{0} > t] = \nu(dy) \quad \text{for every $t \geq 0$}. \label{eq501}
\end{align}
Here we concentrate on the {case where the boundary $\ell$ is inaccessible}:
\begin{align}
	\lim_{z \to \ell} \bP_{x}[\tau_{z}^{+} < \tau_{0}] = 0 \quad (x \in I_{< \ell}).  \label{inaccessible}
\end{align}
The {accessible case i.e., $\ell \in I$ and $\bP_{x}[\tau_{\ell} < \tau_{0}] > 0 \ (x \in (0,\ell))$ will be discussed in Appendix \ref{appendix:twoSideExit}.}
Note that {when \eqref{inaccessible} holds and} $\ell \in I$, it follows from {Remark \ref{rem:W(x,x)m(x)=0}} that $\bP_{\ell}[\tau_{\ell} < \tau_{0}] = 0$ and thus $\ell$ is {irregular} and thus $m(\{\ell \}) = 0$ {again} from Remark \ref{rem:W(x,x)m(x)=0}.

{In order to state our main results, we prepare several conditions and definitions.}
Define for $x,y \in I$ with $x < y$
\begin{align}
	\lambda_{0}^{[x]}(y) := \sup \{ \lambda \geq 0 \mid \bE_{y}[\mathrm{e}^{\lambda \tau_{x}^{-}}] < \infty \}. \label{eq503a}
\end{align}
Note that it obviously holds $\lambda_{0}^{[x]}(y) \leq \lambda^{[x,y]}_{0}$.
For $x \in I_{< \ell}$, we introduce an irreducibility condition of the process $X_{\cdot \wedge \tau_{x}}$ starting from $y \in I_{>x}$:
\begin{align}
	\mathrm{(I)}_{x}: \quad \bP_{y}[\tau_{z} < \tau_{x}] > 0 \quad \text{for every $y \in I_{> x}$ and $z \in (x,\ell)$}. 
	%{[y, z \in (x, \ell) \text{ with }y<z \text{ is better?}]}
	\label{irreducibility}
\end{align}
In this section, we always assume $\mathrm{(I)}_{0}$ holds, but not necessarily for $\mathrm{(I)}_{x}$ for $x \in I_{>0}$.
If $\mathrm{(I)}_{x}$ holds, we easily see from the Markov property that $\lambda_{0}^{[x]} = \lambda_{0}^{[x]}(y)\ (x \in I_{<\ell}, \ y \in I_{>x})$ for
\begin{align}
	\lambda_{0}^{[x]} := \sup \{ \lambda \geq 0 \mid \bE_{u}[\mathrm{e}^{\lambda \tau_{x}^{-}}] < \infty \quad \text{for every $u \in I_{>x}$}\}. %{[\text{I think $I_{>x}$ is right.}]}\label{}
\end{align}
We denote $\lambda_{0} := \lambda_{0}^{[0]}$ {and call it the \textit{decay parameter}.} 
If there is a QSD $\nu$, from the Markov property we easily see that $\bP_{\nu}[\tau_{0} \in dt]$ is exponentially distributed.
Indeed, from \eqref{eq501}, we have $\bP_{\nu}\sbra{\tau_{0} - t \in ds \mid \tau_{0} > t} = \bP_{\nu}\sbra{\tau_{0} \in ds}$ for every $t>0$, which is the lack of memory property.
Hence, the positivity of $\lambda_{0}$ is necessary for the existence of a QSD.
Especially, 
\begin{align}
	\bP_{x}[\tau_{0} < \infty] = 1 \quad (x \in I). \label{certainlyHitZero}
\end{align}
is necessary.
We always assume \eqref{certainlyHitZero} in this section.
The condition $\mathrm{(I)}_{x}$ ensures that $\lambda_{0}^{{[x]}}$ is finite as we show in the following proposition.
%{[I think we should write why the following proposition is necessary. How about the following?]Under the condition $\mathrm{(I)}_{0}$, we have $\lambda_{0}^{{[x]}}<\infty$ for $x \in I$. This fact is ensured from the following proposition.}
%{[It seems that the condition $\mathrm{(I)}_{0}$ not necessarily implies $\lambda_{0}^{[x]} < \infty$.]}

\begin{Prop} \label{prop:finiteLambda0}
	Let $x \in I_{<\ell}$.
	If $\mathrm{(I)}_{x}$ holds,
	then $\lambda_{0}^{[x]} < \infty$.
\end{Prop}

\begin{proof}
	Take $y,z \in (x,\ell)$ with $y < z$.
	Set 
	\begin{align}
		\tau :=
			\inf \{ t > \tau_{z}^{+} \mid X_{t} = y \}, \label{}
	\end{align}
	where we interpret $\inf \emptyset = \infty$ as usual.
	It holds $\bP_{y}[\tau < \infty] > 0$ from $\mathrm{(I)}_{x}$.
	By the right-continuity of $X$, we can take $r > 0$ such that $\delta := \bP_{y}[r < \tau < \tau_{x}] \in (0,1)$.
	Then it holds for $t > r$
	\begin{align}
		\bP_{y}[\tau_{x} > t] &\geq \bP_{y}[\tau_{x} > t, \tau < \tau_{x}] \label{} \\
		&\geq \int_{(r, \infty)}\bP_{y}[\tau_{x}> t-s] \bP_{y}[\tau \in ds, \tau < \tau_{x}] \label{} \\
		&\geq \delta \bP_{y}[\tau_{x} > t-r]. \label{}
	\end{align}
	Inductively, we see for $t > r$ that
	{
	\begin{align}
		\bP_{y}[\tau_{x} > t] \geq \bP_{y}[\tau_{x} > r]\mathrm{e}^{t (\log \delta) /r} \geq \delta \mathrm{e}^{t (\log \delta) /r}, \label{511}
	\end{align}
	}
	and it follows $\lambda_{0}^{[x]} < \infty$.
\end{proof}

We introduce a classification of the boundary $\ell$.

\begin{Def} \label{def:boundaryClassification}
We say that the boundary $\ell$ is \textit{entrance} when
\begin{align}
	\int_{(0,\ell)}W(0,u) m(du) < \infty, \label{entrance}
\end{align}
and we say that the boundary $\ell$ is \textit{non-entrance} when
\begin{align}
	\int_{(0,\ell)}W(0,u) m(du) = \infty. \label{non-entrance}
\end{align}
Note that when $\ell \in I$, the boundary $\ell$ is always entrance from Proposition \ref{prop:integrabilityOfW}.
\end{Def}

Our main objective in this section is to prove the following:
\begin{Thm} \label{thm:existenceOfQSD}
	Let $\cQ$ be the set of QSDs.
	The following holds:
	\begin{enumerate}
		\item If $\ell$ is entrance, then $\lambda_{0} > 0$ and $\cQ = \{ \nu_{\lambda_{0}} \}$.
		\item If $\ell$ is non-entrance and $\lambda_{0} > 0$, then $\cQ = \{ \nu_{\lambda} \}_{\lambda \in (0,\lambda_{0}]}$.
	\end{enumerate}
	Here 
	\begin{align}
		\nu_{\lambda}(dx) := \lambda W^{(-\lambda)}(0,x)m(dx) \quad (\lambda \in (0,\lambda_{0}]). \label{} 
	\end{align}
\end{Thm}
The proof of Theorem \ref{thm:existenceOfQSD} is given as a consequence of the various propositions we see below.
\begin{Rem}\label{Rem503}
	Our classification of the boundary $\ell$ is a generalization of Feller's for one-dimensional diffusions.
	For general results on one-dimensional diffusions we use in this remark, see, e.g., \cite[Chapter 5]{Ito_essentials}, \cite[Chapter V.7]{RogersWilliams2} and \cite[Chapter 33]{Kallenberg-third}.
	Suppose $X$ is a $\frac{d}{dw}\frac{d}{ds}$-diffusion on $I$, that is, a non-singular diffusion on $I$ with the speed measure $dw$ and the (usual) scale function $s$.
	By \eqref{certainlyHitZero}, we may take the scale function $s$ so that $s(0) = 0$ without loss of generality.
	Note that the reference measure $m$ can be taken as $dw$.
	The boundary $\ell$ is said to be entrance or non-entrance depending on $\int_{(0,\ell)}s(x)dw(x)$ is finite or infinite, respectively.
	Since it holds from a well-known potential formula for diffusions that 
	\begin{align}
		\bE_{x}[\tau_{0}] = \int_{(0,\ell)}s(x \wedge y)dw(y). \label{}
	\end{align}
	By the monotone convergence theorem, we have
	\begin{align}
		\int_{(0,\ell)}s(y)dw(y) = \sup_{x \in I_{>x}}\bE_{x}[\tau_{0}] = \int_{(0,\ell)} W(0,y)dw(y), \label{}
	\end{align}
	where the second equality follows from Proposition \ref{prop:boundaryClassificationEquivalence} we will show.
	Thus, our classification is a generalization of Feller's.
	See also \cite[Remark 3.3]{NobaGeneralizedScaleFunc}.
\end{Rem}

We present a probabilistic interpretation of the boundary classification.
The proof will be give after Remark \ref{rem:non-negativityOfPotentialDensity}.
\begin{Prop} \label{prop:boundaryClassificationEquivalence}
	It holds
	\begin{align}
		\sup_{x \in I_{>0}} \bE_{x}[\tau_{0}] = \int_{(0,\ell)}W(0,u)m(du). \label{}
	\end{align}
\end{Prop}

Define 
\begin{align}
	g^{(q)}(x,y) := \bE_{y}[\mathrm{e}^{-q\tau_{x}^{-}}] \quad (q\in[0,\infty), \ x, y\in I ). \label{eq71}
\end{align} 
Note that this function is finite when $q \in (-\lambda_{0}^{[x]}(y),\infty) \cup \{ 0 \}$ and $x,y \in I$ with $x<y$.
We especially write $g^{(q)}(y) := g^{(q)}(0,y)$.
By \eqref{potentialDensityFormula}, \eqref{inaccessible} and taking limit as $z \to \ell$ in \eqref{potentialDensity}, we have {a potential density formula on $I$ killed at $0$:}
\begin{align}
	\int_{0}^{\infty}\mathrm{e}^{-qt} \bP_{x}[X_{t} \in dy, \tau_{0} > t]dt = r^{(q)}(x,y) m(dy) \quad (q \geq 0,\ x,y \in I), \label{potentialDensityOneside-c}
\end{align}
where
\begin{align}
	r^{(q)}(x,y) := g^{(q)}(x) W^{(q)}(0,y) - W^{(q)}(x,y) \quad (q \geq 0,\ x,y \in I). \label{}
\end{align}
By the analytic extension, the equality \eqref{potentialDensityOneside-c} can be extended to $q \in \bC$ with $\Re q > -\lambda_{0}$.

\begin{Rem} \label{rem:non-negativityOfPotentialDensity}
	It holds $r^{(q)}(x,y) > 0$ for every $q \in (-\lambda_{0},\infty) \cup \{0\}$ and $x,y \in I_{>0}$.
	Indeed, we have from Proposition \ref{prop:exitProblem1}
	\begin{align}
		r^{(q)}(x,y) 
		=&g^{(q)}(x) W^{(q)}(0,y) - \bE_x\sbra{e^{-q\tau^-_0}; \tau^-_0<\tau^+_y  }W^{(q)}(0,y) \label{}
		\\
		=& W^{(q)}(0,y) \bE_{x}[\mathrm{e}^{-q \tau_{0}}, \tau_{y}^{+} < \tau_{0}] \in (0,\infty),\label{}
	\end{align}
	where the positivity of the last term follows from Proposition \ref{prop:positivityOfW}, \eqref{certainlyHitZero} and $\mathrm{(I)}_{0}$.
\end{Rem}

We prove Proposition \ref{prop:boundaryClassificationEquivalence}.

\begin{proof}[Proof of Proposition \ref{prop:boundaryClassificationEquivalence}]
	By the potential formula \eqref{potentialDensityOneside-c} for $q = 0$, it holds for $x \in I_{>0}$
	\begin{align}
		\bE_{x}[\tau_{0}] = \int_{(0,\ell)} (W(0,u) - W(x,u)) m(du). \label{}
	\end{align}
	Taking the limit as $x \to \ell$, we have the desired consequence by the monotone convergence theorem.
\end{proof}
\begin{Rem}
	When \eqref{inaccessible} and the entrance condition hold and $\ell \not \in I$, under a suitable regularity of the semigroup of $X$, we can naturally add $\ell$ to the state space $I$ as an instantaneous entrance boundary and consider the law starting from $\ell$ (see \cite{entranceBdry} for details).
\end{Rem}
In the following, we will give several propositions for the proof of Theorem \ref{thm:existenceOfQSD}.
The majority of them can be proven almost as well in the downward skip-free  Markov chain case studied in \cite{QSD_downward_skip-free}.
Therefore, we sometimes omit the proofs and just refer to the corresponding ones in \cite{QSD_downward_skip-free}.
We first show that every QSD is represented by a scale function.

\begin{Lem} [{cf. \cite[Lemma 4.1]{QSD_downward_skip-free}}]\label{lem:qsdCharacterization-c}
	Let $\nu$ be a QSD such that $\bP_{\nu}[\tau_{0} > t] = \mathrm{e}^{-\lambda t}$ for some $\lambda > 0$.
	Then it holds
	\begin{align}
		\nu(dx) = \lambda W^{(-\lambda)}(0,x) m(dx) \quad (x \in I). \label{eq51}
	\end{align}
\end{Lem}

\begin{proof}
	On the one hand, from \eqref{potentialDensityOneside-c} for $q = 0$, it holds
	\begin{align}
		\int_{0}^{\infty}\bP_{\nu}[X_{t} \in dy, \tau_{0} > t]dt &= \left(\int_{I}r^{(0)}(x,y)\nu(dx) \right) m(dy) \label{} \\
		&= \left( W(0,y) - \int_{I} W(x,y)\nu(dx) \right) m(dy). \label{eq15}
	\end{align}
	On the other hand, since $\nu$ is a QSD with $\bP_{\nu}[\tau_{0} > t] = \mathrm{e}^{-\lambda t}$, we have
	\begin{align}
		\int_{0}^{\infty}\bP_{\nu}[X_{t} \in dy, \tau_{0} > t]dt 
		&=\int_{0}^{\infty}\bP_{\nu}[X_{t} \in dy \mid \tau_{0} > t]\bP_\nu[\tau_{0} > t]dt \\
		&= \nu(dy)\int_{0}^{\infty}\mathrm{e}^{-\lambda t}dt = \frac{\nu(dy)}{\lambda}. \label{eq16}
	\end{align}
	From \eqref{eq15} and \eqref{eq16}, we see $\nu$ is absolutely continuous w.r.t.\ $m$.
	Let us denote the density by $\rho$.
	Then it follows
	\begin{align}
		\rho(y) = \lambda W(0,y) - \lambda \int_{(0,y)}\rho(u) W(u,y) m(du) \quad m\text{-a.e.\  $y \in I$}. \label{}
	\end{align}
	From Corollary \ref{Cor404}, we see $\rho(y) = \lambda W^{(-\lambda)}(0,y)$ $m$-a.e.
\end{proof}

{
Since we have known the possible form of a QSD, it suffices to consider when the measure of that form is a QSD.
We introduce a function $h^{(q)}$ that plays an essential role to characterize the existence of QSDs.
}

\begin{Prop}[{cf. \cite[Proposition 4.2]{QSD_downward_skip-free}}] \label{prop:nonnegativityImpliesIntegrability-entrance-c}
	Let $x \in I_{<\ell}$ and $\lambda \in (-\infty,\lambda_{0}^{[x]}] \cap \bR$.
	For $q > -\lambda$, the limit
	\begin{align}
		h^{(q)}(\lambda;x) := \lim_{y \to \ell} \frac{W^{(-\lambda)}(x,y)}{W^{(q)}(x,y)} \label{defOfH-c}
	\end{align}
	exists (we especially write $h := h^{(0)}$) and finite.
	% If $\lambda < \lambda_{0}$ or $\lambda \leq 0$ it also holds
	% \begin{align}
	% 	h^{(q)}(\lambda;x) = \frac{g^{(-\lambda)}(x)}{g^{(q)}(x)}h^{(q)}(\lambda;0). \label{eq34-c}
	% \end{align}
	In addition, it holds
	\begin{align}
		0 < \int_{I_{>x}}  W^{(-\lambda)}(x,u) g^{(q)}(x,u) m(du) \leq \frac{1}{\lambda+q} \quad (x \in I)
		 \label{eq44-c}
	\end{align}
	and
	\begin{align}
		h^{(q)}(\lambda;x) = 1 - (\lambda+q) \int_{I_{>x}} W^{(-\lambda)}(x,u)g^{(q)}(x,u) m(du). \label{h-funcRepresentation-c}
	\end{align}
	The function $h^{(q)}$ has the following properties:
	\begin{enumerate}
		\item $0 \leq h^{(q)}(\lambda;x) < 1$.
		\item $h^{(q)}(\lambda;x)$ is non-increasing in $\lambda$ and $q$.
		\item For fixed $q > 0$, the RHS of \eqref{h-funcRepresentation-c} can be analytically extended in $\lambda \in \bC$ with $|\lambda| < q$.
		\item For fixed $\lambda \in (-\infty,\lambda_{0}^{[x]})$, the RHS of \eqref{h-funcRepresentation-c} can be analytically extended in $q \in \bC$ with $\Re q > -\lambda$.
		\item If $h^{(q)}(\lambda;x) = 0$ for some $\lambda \in (-\infty, \lambda_{0}^{[x]})$ and $q \in (-\lambda,\infty)$,
		it also holds for every $\lambda \in (-\infty, \lambda_{0}^{[x]}] \cap \bR$ and $q \in ( -\lambda, \infty)$.
	\end{enumerate}
\end{Prop}

\begin{proof}
	Let $x \in I_{<\ell}$ and take $\lambda \in (-\infty,\lambda_{0}^{[x]}] \cap \bR$ and $q \in (-\lambda,\infty)$.
	We have for $y \in (x, \ell)$
	\begin{align}
		0 &\leq \frac{W^{(-\lambda)}(x,y)}{W^{(q)}(x,y)} \label{} \\
		&=
		\frac{1}{W^{(q)}(x,y)}\rbra{ W^{(q)}(x,y) -  (\lambda+q) \int_{(x,y)}W^{(-\lambda)}(x,u)W^{(q)}(u,y)m(du)}
		\\
		&=
		1 - (\lambda + q) \int_{(x,y)} W^{(-\lambda)}(x,u) \bE_{u}[\mathrm{e}^{-q\tau_{x}},\tau_{x} < \tau_{y}^{+}] m(du), \label{}
	\end{align}
	where we used, in the first inequality, Proposition \ref{prop:positivityOfW} and Remark \ref{Rem502}, in the first equality Propositions \ref{prop:resolventIdentity}, and in the second equality, the analytic extension of \eqref{eq56} to $q \in \bC$ with $\Re q > -\lambda_{0}^{[x]}$.
	Taking the limit as $y \to \ell$, we have from the monotone convergence theorem and \eqref{inaccessible}
	\begin{align}
		0 \leq \lim_{y \to \ell} \frac{W^{(-\lambda)}(x,y)}{W^{(q)}(x,y)} = 1 - (\lambda+q) \int_{(x,\ell)} W^{(-\lambda)}(x,u) g^{(q)}(x,u) m(du) < 1. \label{eq18-entrance-c}
	\end{align}
	Here we note that $W^{(-\lambda)}(x,u) > 0$ for $u \in (x,x + \delta)$ for sufficiently small $\delta > 0$ from Corollary \ref{cor:ratioLimitScaleFunc} and this fact implies the last inequality of \eqref{eq18-entrance-c}. 
	Thus, since $m$ does not have a mass at $\ell$, we obtain \eqref{eq44-c}, \eqref{h-funcRepresentation-c} and the property (i). 
	The property (ii) is clear from Proposition \ref{prop:positivityOfW} and Remark \ref{Rem502}.
	% If $\lambda < \lambda_{0}$ or $\lambda \leq 0$, by analytically extending Proposition \ref{prop:exitProblem1} to $q \in \bC$ with $\Re q > -\lambda_0 $, we have
	% \begin{align}
	% 	h^{(q)}(\lambda;x) = \lim_{y \to \ell-} \frac{W^{(-\lambda)}(x,y)}{W^{(-\lambda)}(0,y)} \frac{W^{(-\lambda)}(0,y)}{W^{(q)}(0,y)}\frac{W^{(q)}(0,y)}{W^{(q)}(x,y)} = \frac{g^{(-\lambda)}(x)}{g^{(q)}(x)}h^{(q)}(\lambda;0). \label{}
	% \end{align}
	We show (iii).
	Let $\zeta \in \bC$.
	For $x,y \in I$ with $x < y$, it holds from \eqref{eq01} 
	\begin{align}
		\left|\frac{d}{d\zeta}W^{(-\zeta)}(x,y)\right| 
		&\leq \sum_{n \geq 0}(n+1)|\zeta|^{n} W^{\otimes (n+2)}(x,y) \label{} \\
		&\leq \left( \int_{(x,y)}W(x,v) m(dv) \right) W(x,y) \mathrm{e}^{|\zeta|\int_{(x,y)}W(x,v) m(dv)}. \label{}
	\end{align}
	Then it follows for every $c \in (x, \ell)$
	\begin{align}
		&\int_{(x,c)} \left| \frac{d}{d\zeta} W^{(-\zeta)}(x,u)\right| g^{(q)}(x,u) m(du) \label{} \\ 
		\leq &\left( \int_{(x,c)}W(x,v) m(dv) \right) \mathrm{e}^{|\zeta|\int_{(x,c)}W(x,v) m(dv)} \int_{(x,c)} W(x,u) g^{(q)}(x,u) m(du) < \infty. \label{}
	\end{align}
	Thus, the function
	\begin{align}
		\bC \ni \zeta \longmapsto \int_{(x,c)} W^{(-\zeta)}(x,u) g^{(q)}(x,u) m(du) \label{}
	\end{align}
	is holomorphic on $\bC$.
	Take $q > 0$ and $ r \in (0,q)$.
	Let $|\zeta| \leq r$.
	Since it holds $|W^{(-\zeta)}(x,y)| \leq W^{(r)}(x,y)$ from Proposition \ref{prop:defOfM},
	we see from \eqref{eq18-entrance-c} that for every $x, c \in [0,\ell)$ with $x < c$
	\begin{align}
		\int_{(c,\ell)} | W^{(-\zeta)}(x,u)| g^{(q)}(x,u) m(du) \leq \int_{(c,\ell)} W^{(r)}(x,u) g^{(q)}(x,u) m(du) \leq \frac{1}{q-r}. \label{}
	\end{align}
	Thus, the uniform convergence
	\begin{align}
		\lim_{c \to \ell} \sup_{|\zeta| \leq r } \left|\int_{(c,\ell)} W^{(-\zeta)}(x,u) g^{(q)}(x,u) m(du) \right| = 0 \label{}
	\end{align}
	holds, and the function $\zeta \mapsto \int_{(x,\ell)} W^{(-\zeta)}(x,u) g^{(q)}(x,u) m(du)$ is analytic on $|\zeta| < q$ by Montel's theorem.
	The property (iv) follows similarly.
	Indeed, let $\lambda \in (-\infty,\lambda_{0}^{[x]})$, $q \in (-\lambda,\infty)$ and $r \in (-\lambda,q)$.
	For $\zeta \in \bC$ and $x \in I_{< \ell}$ with $\Re \zeta > q$ it holds for $c \in [x,\ell)$
	\begin{align}
		&\int_{(c,\ell)} W^{(-\lambda)}(x,u) \left|\frac{d}{d\zeta} g^{(\zeta)}(x,u)\right| m(du) \label{} \\
		\leq &C \int_{(c,\ell)} W^{(-\lambda)}(x,u) g^{(q)}(x,u) m(du) \leq \frac{C}{q + \lambda}, \label{}
	\end{align}
	where $C := \sup_{x \geq 0} x\mathrm{e}^{-(r-q) x}$.
	{
	The proof of the assertion (v) is exactly the same as \cite[Proposition 4.2]{QSD_downward_skip-free}, and we omit it.
	}
\end{proof}

From \eqref{h-funcRepresentation-c} for $q = 0$, we see the following.

\begin{Cor}[{cf. \cite[Corollary 4.3]{QSD_downward_skip-free}}] \label{cor:ScaleDefinesFiniteMeasure-c}
	Let $x \in I_{< \ell}$ and suppose $\lambda_{0}^{[x]} > 0$.
	For $\lambda \in (0, \lambda_{0}^{[x]}] \cap \bR$, it holds
	\begin{align}
		\lambda \int_{(x,\ell)} W^{(-\lambda)}(x,u) m(du) = 1 - h (\lambda;x) \in (0,1], \label{}
	\end{align}
	which also implies $h(\lambda;x) = \lim_{y \to \ell}Z^{(-\lambda)}(x,y)$.
\end{Cor}

As another corollary, we show the positivity of $W^{(-\lambda_{0})}(x,\cdot)$ under $\mathrm{(I)}_{x}$.

\begin{Cor}[{cf. \cite[Corollary 4.4]{QSD_downward_skip-free}}] \label{cor:non-negativityOfScale-entrance-c}
	Let $x \in I_{<\ell}$ and suppose $\mathrm{(I)}_{x}$ holds.
	Then it holds
	\begin{align}
		W^{(-\lambda_{0}^{[x]})}(x,y) > 0 \quad (y \in (x,\ell)). \label{}
	\end{align}
\end{Cor}

\begin{proof}
	From Proposition \ref{prop:resolventIdentity} and Corollary \ref{cor:ScaleDefinesFiniteMeasure-c},
	we have for $y \in (x,\ell)$
	\begin{align}
		&W^{(-\lambda_{0}^{[x]})}(x,y) \label{} \\
		= &W(x,y) - \lambda_{0}^{[x]} \int_{(x,y)} W^{(-\lambda_{0}^{[x]})}(x,u)W(u,y) m(du) \label{} \\
		= &\frac{\lambda_{0}^{[x]}W(x,y)}{1 - h(\lambda_{0}^{[x]};x)}\int_{(x,\ell)} W^{(-\lambda_{0}^{[x]})}(x,u) m(du) - \lambda_{0}^{[x]} \int_{(x,y)} W^{(-\lambda_{0}^{[x]})}(x,u)W(u,y) m(du) \label{} \\
		= &\frac{\lambda_{0}^{[x]} W(x,y)}{1 - h(\lambda_{0}^{[x]};x)} \int_{(x,\ell)} W^{(-\lambda_{0}^{[x]})}(x,u)
		\left( 1 - (1 - h(\lambda_{0}^{[x]};x)) \bP_{u}[\tau_{x} < \tau_{y}^{+}] \right) m(du), \label{eq57} 
	\end{align}
	where we used Proposition \ref{prop:exitProblem1} in the last equality.
	Since $\bP_{u}[\tau_{x} < \tau_{y}^{+}] < 1 \ (u \in (x,y))$ from $\mathrm{(I)}_{x}$, and the function $y \mapsto W^{(-\lambda_{0}^{[x]})}(x,y)$ is positive on some measurable set $A \subset (x,\ell)$ with $m(A) > 0$ from Corollary \ref{cor:ScaleDefinesFiniteMeasure-c},
	we see that \eqref{eq57} is strictly positive.
\end{proof}

The following lemma shows that existence of a QSD is determined whether $h^{(q)}(\lambda;0)$ is zero or not, {whose proof is essentially the same as \cite[Lemma 4.5]{QSD_downward_skip-free}, and we omit it.}

\begin{Lem}[{cf. \cite[Lemma 4.5]{QSD_downward_skip-free}}] \label{lem:scaleFuncGivesQSD-entrance-c}
	Suppose $\lambda_{0} > 0$.
	For $\lambda \in (0,\lambda_{0}]$, the (sub)probability measure
	\begin{align}
		\nu_{\lambda}(dx) := \lambda W^{(-\lambda)}(0,x) m(dx) \quad (x \in I) \label{}
	\end{align}
	is a QSD with $\bP_{\nu_{\lambda}}[\tau_{0} > t] = \mathrm{e}^{-\lambda t} \ (t \geq 0)$ if and only if $h^{(q)}(\lambda;0) = 0$ for some $q \geq 0$.
	It is also equivalent to 
	\begin{align}
		\nu_{\lambda}(0,\ell) = 1. \label{eq19-c}
	\end{align}
\end{Lem}

From Lemmas \ref{lem:qsdCharacterization-c} and \ref{lem:scaleFuncGivesQSD-entrance-c},
to see the existence of QSDs, it is enough to consider the zeros of the function $h$.
The following lemma shows that the non-entrance condition implies the function $h$ is the zero function.

\begin{Lem}[{cf.\cite[Lemma 4.6]{QSD_downward_skip-free}}] \label{lem:BdryClassificationEquivalence-c}
	Suppose $\lambda_{0} > 0$. 
	The boundary $\ell$ is non-entrance if and only if $h(\lambda;x) = 0$ for some (or equivalently, every) $x \in I_{<\ell}$ and some (or equivalently, every) $\lambda \in (0,\lambda_{0}^{[x]})$.
\end{Lem}

\begin{proof}
	Suppose the boundary $\ell$ is non-entrance.
	Let $x \in I_{<\ell}$ and take $\lambda \in (0,\lambda_{0}^{[x]})$.
	If $h(\lambda;x) \neq 0$, it follows from \eqref{defOfH-c} for $q = 0$
	\begin{align}
		h(\lambda;x) = \lim_{y \to \ell}\frac{W^{(-\lambda)}(x,y)}{W(x,y)} \in (0,1). \label{}
	\end{align}
	Since it holds $\int_{(0,\ell)} W^{(-\lambda)}(x,y) m(du) < \infty$ and $\lim_{y \to \ell} W(x,y) / W(0,y) = 1$ from Corollary \ref{cor:ScaleDefinesFiniteMeasure-c} and \eqref{inaccessible}, respectively,
	it follows $\int_{(0,\ell)} W(0,u) m(du) < \infty$ and it contradicts to the non-entrance condition.

	For $x \in I_{<\ell}$ and $\lambda' \in (0,\lambda_{0}^{[x]})$, suppose $h(\lambda';x) = 0$.
	From Proposition \ref{prop:nonnegativityImpliesIntegrability-entrance-c} (v) it also holds for every $\lambda \in (0,\lambda_{0}^{[x]}]$.
	From Proposition \ref{prop:resolventIdentity}, it holds $W^{(-\lambda)}(x,u) < W(x,u)$ for $u \in I_{>x}$.
	From Lemma \ref{lem:scaleFuncGivesQSD-entrance-c} we have
	\begin{align}
		\frac{1}{\lambda} = \int_{(x,\ell)} W^{(-\lambda)}(x,u) m(du) \leq \int_{(x,\ell)} W(x,u) m(du) \leq \int_{(0,\ell)} W(0,u) m(du) . \label{eq36-c}
	\end{align}
	Taking limit as $\lambda \to 0+$, we obtain the non-entrance condition.
\end{proof}

{
Combining Lemmas \ref{lem:qsdCharacterization-c}, \ref{lem:scaleFuncGivesQSD-entrance-c} and \ref{lem:BdryClassificationEquivalence-c},
we easily see Theorem \ref{thm:existenceOfQSD} (ii) follows.
We omit the proof.
We focus on the proof of Theorem \ref{thm:existenceOfQSD} (i), that is, the case where $\ell$ is entrance.}
Under the entrance condition, the function $W^{(q)}(x,\cdot)$ is $m$-integrable for every $x \in I_{<\ell}$.

\begin{Lem} [{cf. \cite[Lemma 4.7]{QSD_downward_skip-free}}]\label{lem:analyticExtentionOfh}
	Suppose the boundary $\ell$ is entrance.
	Then it holds for every $R > 0$ and $x \in I_{<\ell}$
	\begin{align}
		\sup_{\zeta \in \bC, |\zeta| \leq R} \int_{(x,\ell)} |W^{(\zeta)}(x,u)| m(du) < \infty. \label{eq21}
	\end{align}
	In addition, the limit $Z^{(q)}(x) := \lim_{y \to \ell}Z^{(q)}(x,y) \ (q \in \bC, x \in I)$
	exists and
	\begin{align}
		Z^{(q)}(x) = 1 + q \int_{(x,\ell)} W^{(q)}(x,u) m(du) \quad (q \in \bC, \ x \in I). \label{eq50-c}
	\end{align}
	Hence, the function $q \mapsto Z^{(q)}(x) \ (x \in I)$ is an entire function.
	Moreover, it holds
	\begin{align}
		Z^{(q)}(x) > 0 \quad (x \in I, \ q > -\lambda_{0}^{[x]}) \label{eq70}
	\end{align}
	and
	\begin{align}
		g^{(q)}(x,y) = \frac{Z^{(q)}(y)}{Z^{(q)}(x)} \quad (q > -\lambda_{0}^{[x]},\  y \in I_{\geq x}), \label{eq53}
	\end{align}
	and thus the function $q \mapsto g^{(q)}(x,y) \ (y \in I_{\geq x})$ is analytically extended to a meromorphic function on $\bC$. %{[No. The RHS is meromorphic on $\bC$ (because $Z^{(q)}$ is an entire function) and thus we can extend $q \mapsto g^{(q)}(x,y)$ to a meromorphic function on $\bC$.]}
\end{Lem}

{
The proof of Lemma \ref{lem:analyticExtentionOfh} is almost the same as \cite[Lemma 4.7]{QSD_downward_skip-free} and we omit it.
}

% \begin{proof}
% 	From \eqref{eq03} and the entrance condition,
% 	we easily see \eqref{eq21} and \eqref{eq50-c} hold.
% 	From Proposition \ref{prop:exitProblemZ}, it holds for $x,y,z \in I_{\geq 0}$ with $x < y < z$ and $q \geq 0$
% 	\begin{align}
% 		\frac{Z^{(q)}(y,z)}{Z^{(q)}(x,z)} = \frac{W^{(q)}(y,z)}{W^{(q)}(x,z)} + \frac{\bE_{y}[\mathrm{e}^{-q\tau_{z}^{+}},\tau_{z}^{+} < \tau_{x}]}{Z^{(q)}(x,z)}. \label{}
% 	\end{align}
% 	From \eqref{inaccessible}, we have
% 	\begin{align}
% 		g^{(q)}(x,y) = \lim_{z \to \ell} \frac{W^{(q)}(y,z)}{W^{(q)}(x,z)} = \lim_{z \to \ell} \frac{Z^{(q)}(y,z)}{Z^{(q)}(x,z)} = \frac{Z^{(q)}(y)}{Z^{(q)}(x)}. \label{eq55}
% 	\end{align}
% 	Let $x \in I_{< \ell}$ and suppose $\lambda_{0}^{[x]} > 0$.
% 	By Corollary \ref{cor:ScaleDefinesFiniteMeasure-c} and Lemma \ref{lem:BdryClassificationEquivalence-c} we see $Z^{(-\lambda)}(x) = h(\lambda;x) > 0$ for $\lambda \in (0,\lambda_{0}^{[x]})$.
% 	Since the both sides of \eqref{eq55} is analytic for $q \in (-\lambda_{0}^{[x]},0)$, we obtain \eqref{eq53} by the identity theorem.
% \end{proof}

\begin{Rem}
	When $\ell \in I$, it clearly holds $Z^{(q)}(x) = Z^{(q)}(x,\ell)$.
\end{Rem}

The following corollary is immediate from \eqref{eq53}.

\begin{Cor} \label{cor:strictlyIncreasingZ}
	Assume the boundary $\ell$ is entrance.
	For $x \in I_{<\ell}$ and $q \in  (-\lambda_{0}^{[x]},0)$, the function $I_{>x}  \ni y \mapsto Z^{(q)}(y) \in (0,1)$ is strictly increasing.
\end{Cor}

The following lemma shows that under the entrance condition, it always holds $\lambda_{0} > 0$.
{
We omit the proof since it is essentially the same as \cite[Lemma 4.9]{QSD_downward_skip-free}.
}

\begin{Lem} [cf. {\cite[Proposition 3]{BertoinQSD} and \cite[Lemma 4.9]{QSD_downward_skip-free}}]\label{lem:r-recurrence} 
	Assume the boundary $\ell$ is entrance. 
	Then for every $x,y \in I$ with $x < y$, it holds $\lambda_{0}^{[x]}(y) > 0$. 
	Suppose in addition that $\lambda_{0}^{[x]}(y) < \infty$.
	Then the function $\bC \ni q \mapsto Z^{(q)}(y) / Z^{(q)}(x)$ has a pole at $q = -\lambda_{0}^{[x]}(y)$,
	and it is one with the minimum absolute value.
	In particular, it holds
	\begin{align}
		\lim_{q \to -\lambda_{0}^{[x]}(y)+} g^{(q)}(x,y) = \infty \quad (y \in I_{>x}). \label{eq37}
	\end{align}
\end{Lem}

% \begin{proof}
% 	Fix $x,y \in I$ with $x < y$ and suppose $\lambda_{0}^{[x]}(y) < \infty$.
% 	From the definition of $\lambda_{0}^{[x]}(y)$, 
% 	\begin{align}
% 		\text{the function } \frac{Z^{(q)}(y)}{Z^{(q)}(x)} \text{ does not have a pole at $q \in \{ \zeta \in \bC \mid \Re \zeta > -\lambda_{0}^{[x]}(y) \}$}. \label{eq35}
% 	\end{align}
% 	Suppose there exists a pole and let $\zeta'$ be the one with the minimum absolute value.
% 	Note that $|\zeta'| > 0$ since $Z^{(0)}(y) / Z^{(0)}(x) = 1$.
% 	We show $\zeta' = -\lambda_{0}^{[x]}(y)$.
% 	From \eqref{eq35}, we see that $\Re \zeta' \leq -\lambda_{0}^{[x]}(y)$, which implies $|\zeta'| \geq \lambda_{0}^{[x]}(y)$.
% 	Suppose $|\zeta'| > \lambda_{0}^{[x]}(y)$.
% 	Then the function $Z^{(q)}(y) / Z^{(q)}(x)$ has a power series expansion around $0$ for $q \in \bC$ with $|q| < |\zeta'|$ and
% 	the coefficient of $q^{n}$ is given by the $n$-th right-derivative of $\bE_{y}[\mathrm{e}^{-q \tau_{x}}]$ at $q = 0$ divided by $n!$, which is $(-1)^{n}\bE_{y}[(\tau_{x}^{-})^{n}] / n!$.
% 	Since the series converges absolutely for $q \in \bC$ with $|q| < |\zeta'|$, it follows $g^{(-|\zeta'| + \delta)}(x,y) < \infty$ for every $\delta > 0$.
% 	It contradicts to the definition of $\lambda_{0}^{[x]}(y)$, and thus it follows $\zeta' = -\lambda_{0}^{[x]}(y)$.
% 	When we suppose there are no poles,
% 	by the same argument above, we have $\lambda_{0}^{[x]}(y) = \infty$ and that contradicts to our assumption.
% \end{proof}

Considering the case $x = 0$ in Lemma \ref{lem:r-recurrence}, we have the following:

\begin{Cor}[{cf. \cite[Lemma 5.1]{QSD_downward_skip-free}}] \label{cor:lambdaZeroPositivity}
	Assume the boundary $\ell$ is entrance.
	Then it holds $\lambda_{0} \in (0,\infty)$, $h(\lambda_{0};0) = 0$
	and the process $X_{\cdot \wedge \tau_{0}}$ is $\lambda_{0}$-recurrent, that is, for every measurable set $A \subset I_{>0}$ with $m(A) > 0$
	\begin{align}
		\int_{0}^{\infty}\mathrm{e}^{\lambda_{0}t}\bP_{x}[X_{t} \in A, \tau_{0} > t]dt = \infty \quad (x \in I_{>0}). \label{lambdaRecurrence}
	\end{align}
	In addition, we have the following characterization of $\lambda_{0}$:
	\begin{align}
		\lambda_{0} = \min \{ \lambda \geq 0 \mid Z^{(-\lambda)}(0) = 0 \}. \label{eq45}
	\end{align}
\end{Cor}

\begin{proof}
	It directly follows from Lemma \ref{lem:r-recurrence} that $\lambda_{0} > 0$ and $Z^{(-\lambda_{0})}(0) = 0$.
	From \eqref{eq70} and Corollary \ref{cor:ScaleDefinesFiniteMeasure-c}, we have \eqref{eq45} and $h(\lambda_{0};0) = 0$.
	The $\lambda_{0}$-recurrence \eqref{lambdaRecurrence} follows from Remark \ref{rem:non-negativityOfPotentialDensity}.
\end{proof}
{
	Now we can prove Theorem \ref{thm:existenceOfQSD} (i) and complete the proof.
	\begin{proof}[Proof of Theorem \ref{thm:existenceOfQSD} (i)]
		From Lemmas \ref{lem:BdryClassificationEquivalence-c} and \ref{cor:lambdaZeroPositivity}, under the entrance condition it holds $\lambda_{0} \in (0,\infty)$, $h(\lambda;0) > 0 \ (\lambda \in (0,\lambda_{0}))$ and $h(\lambda_{0};0) = 0$.
		Thus, we have $\cQ = \{ \nu_{\lambda_{0}}\}$ by Lemma \ref{lem:scaleFuncGivesQSD-entrance-c}.
	\end{proof}
}

When the boundary $\ell$ is entrance, we can derive a kind of mean ergodicity to the unique QSD $\nu_{\lambda_{0}}$.
For the purpose, we first show the following.

\begin{Prop} \label{prop:lambdaX}
	Assume the boundary $\ell$ is entrance.
	The following holds:
	\begin{enumerate}
		\item The function $Z^{(-\lambda_{0})}$ is $\lambda_{0}$-invariant, that is, $Z^{(-\lambda_{0})}(x) > 0 \ (x \in I_{>0})$ and
		\begin{align}
			\bE_{x}[Z^{(-\lambda_{0})}(X_{t}), \tau_{0} > t] = \mathrm{e}^{-\lambda_{0} t}Z^{(-\lambda_{0})}(x) \quad (t \geq 0, \ x \in I). \label{eq63}
		\end{align}
		\item For $x,y \in I$ with $x < y$, it holds $\lambda_{0}^{[x]}(y) > \lambda_{0}$.
		\item The function $I \ni x \to Z^{(-\lambda_{0})}(x)$ is strictly increasing.
		\item The function $\bC \ni q \mapsto Z^{(q)}(0)$ has a simple root at $q = -\lambda_{0}$ and it holds
		\begin{align}
			\rho := \left. \frac{d}{dq} Z^{(q)}(0) \right|_{q = -\lambda_{0}} = \int_{(0,\ell)}W^{(-\lambda_{0})}(0,u)Z^{(-\lambda_{0})}(u)m(du) \in (0,\infty). \label{}
		\end{align}
	\end{enumerate}
\end{Prop}

\begin{proof}
	From \eqref{potentialDensityOneside-c}, \eqref{eq53} and Corollary \ref{cor:resolventIdentityForZ},
	we have for $q \in (-\lambda_{0},\infty)$ and $r \in [-\lambda_{0},\infty)$ with $q \neq r$ and $x \in I_{<\ell}$
	\begin{align}
		&\int_{0}^{\infty}\mathrm{e}^{-qt} \bE_{x}[Z^{(r)}(X_{t}), \tau_{x} > t] dt \label{} \\
		= &\int_{(0,\ell)} \left( \frac{Z^{(q)}(x)}{Z^{(q)}(0)}W^{(q)}(0,u) - W^{(q)}(x,u) \right) Z^{(r)}(u)m(du) \label{} \\
		= & \lim_{z \to \ell} \left( \frac{Z^{(q)}(x)}{Z^{(q)}(0)} W^{(q)} \otimes Z^{(r)}(0,z) - W^{(q)} \otimes Z^{(r)}(x,z)\right) \label{} \\
		= & \frac{1}{q-r} \left( \frac{Z^{(q)}(x)}{Z^{(q)}(0)}(Z^{(q)}(0) - Z^{(r)}(0)) - (Z^{(q)}(x) - Z^{(r)}(x)) \right) \label{} \\
		= & \frac{1}{q-r} \left( Z^{(r)}(x) - \frac{Z^{(q)}(x)}{Z^{(q)}(0)}Z^{(r)}(0) \right). \label{} 
	\end{align}
	By setting $r = -\lambda_{0}$, we have from \eqref{eq45}
	\begin{align}
		\int_{0}^{\infty}\mathrm{e}^{-qt} \bE_{x}[Z^{(-\lambda_{0})}(X_{t}), \tau_{0} > t] dt = \frac{Z^{(-\lambda_{0})}(x)}{q+\lambda_{0}} \quad (q > -\lambda_{0}). \label{eq64}
	\end{align}
	Since $X$ is right-continuous and the function $I \ni x \mapsto Z^{(-\lambda_{0})}(x)$ is bounded continuous, which easily follows from Proposition \ref{prop:continuityOfW} and the dominated convergence theorem, 
	the equality \eqref{eq63} holds.
	If $Z^{(-\lambda_{0})}(x) = 0$ for some $x \in (0,\ell)$, it follows from Remark \ref{rem:non-negativityOfPotentialDensity}, \eqref{potentialDensityOneside-c} and \eqref{eq64} that $Z^{(-\lambda_{0})}(x) = 0$ for $m$-a.e.\ $x \in (0,\ell)$.
	It is, however, impossible because $\lim_{x \to \ell}Z^{(-\lambda_{0})}(x) = 1$, which can be seen from Proposition \ref{prop:positivityOfW} and the entrance condition as follows:
	\begin{align}
	0 \leq 1 - Z^{(-\lambda_{0})}(x) \leq \lambda_{0} \int_{(x,\ell)}W^{(-\lambda_{0})}(x,u)m(du) \leq \lambda_{0} \int_{(x,\ell)}W(0,u)m(du) \xrightarrow[]{x \to \ell} 0. \label{}
	\end{align}
	Thus, we obtain (i).
	The assertion (ii) follows from Lemma \ref{lem:r-recurrence} and (i).
	From (ii), it holds for every $x,y \in I$ with $x < y$,
	\begin{align}
		Z^{(-\lambda_{0})}(y) = g^{(-\lambda_{0})}(x,y) Z^{(-\lambda_{0})}(x) > Z^{(-\lambda_{0})}(x), \label{}
	\end{align}
	and (iii) holds.
	From Corollary \ref{cor:resolventIdentityForZ} and the dominated convergence theorem,
	we have for $q \in \bC$
	\begin{align}
		Z^{(q)}(0) = Z^{(q)}(0) - Z^{(-\lambda_{0})}(0) = (q + \lambda_{0}) \int_{(0,\ell)}W^{(q)}(0,u)Z^{(-\lambda_{0})}(u)m(du), \label{}
	\end{align}
	and we see (iv) holds.
\end{proof}
{
The following is the result announced above.
}
\begin{Thm} \label{thm:meanYaglomLimit}
	Assume the boundary $\ell$ is entrance.
	For a function $f: (0,\ell) \to \bR$ such that $\int_{(0,\ell)}W^{(-\lambda_{0})}(0,u) |f(u)|m(du) < \infty$, it holds
	\begin{align}
		\begin{split}			
			&\lim_{t \to \infty} \frac{1}{t}\int_{0}^{t}\mathrm{e}^{\lambda_{0} s} \bE_{x}[f(X_{s}), \tau_{0} > s]ds \\
			= &\frac{Z^{(-\lambda_{0})}(x)}{\rho} \int_{(0,\ell)} W^{(-\lambda_{0})}(0,u)f(u)m(du). 
		\end{split}
		\label{meanYaglom}
	\end{align}
\end{Thm}

\begin{proof}
	Take a function $f:(0,\ell) \to \bR$ such that $\int_{(0,\ell)}W^{(-\lambda_{0})}(0,u)|f(u)|m(du) < \infty$.
	From \eqref{potentialDensityOneside-c} and Proposition \ref{prop:lambdaX} (iv), we have for $x \in (0,\ell)$
	\begin{align}
		&\int_{0}^{\infty}\mathrm{e}^{-qt}\bE_{x}[f(X_{t}), \tau_{0} > t]dt \label{} \\
		\sim &\frac{Z^{(-\lambda_{0})}(x)}{\rho (q + \lambda_{0})} \int_{(0,\ell)}W^{(-\lambda_{0})}(0,u)f(u)m(du) \quad (q \to -\lambda_{0}+). \label{}
	\end{align}
	By Karamata's Tauberian theorem (see, e.g., \cite[Theorem 1.7.1]{Regularvariation}), it follows that
	\begin{align}
		&\int_{0}^{t}\mathrm{e}^{\lambda_{0} s} \bE_{x}[f(X_{s}), \tau_{0} > s]ds \label{} \\
		\sim &t \frac{Z^{(-\lambda_{0})}(x)}{\rho} \int_{(0,\ell)} W^{(-\lambda_{0})}(0,u)f(u)m(du) \quad (t \to \infty). \label{}
	\end{align}
	The proof is complete.
\end{proof}

\begin{Rem}
	The limit \eqref{meanYaglom} is, in fact, the mean ergodicity of $X$ under a change of measure of $\bP$.
	From Proposition \ref{prop:lambdaX} (i), we see
	\begin{align}
		M_{t} := \mathrm{e}^{\lambda_{0}t} \frac{Z^{(-\lambda_{0})}(X_{t})}{Z^{(-\lambda_{0})}(X_{0})} 1\{ {\tau_0} > t \} \label{eq80}
	\end{align}
	is a non-negative martingale with $M_{0} = 1$ under $\bP_{x}$ for every $x \in I_{>0}$.
	We denote the change of measure by $M_{t}$ by $\bP^{(\lambda_{0})} := \{ \bP_{x}^{(\lambda
	_{0})} \}_{x \in I_{>0}}$:
	\begin{align}
		\left.\frac{d\bP^{(\lambda_{0})}_{x}}{d\bP_{x}}\right|_{\cF_{t}} := M_{t} \quad (t > 0). \label{}  
	\end{align}
	Note that $X$ does not hit $0$ under $\bP^{(\lambda_{0})}_{x}$ since it obviously follows from \eqref{eq80} that $\bP^{(\lambda_{0})}_{x}[{\tau_0}< t] = 0$ for every $t > 0$.
	Define a measure on $(0,\ell)$
	\begin{align}
		\pi(du) := \frac{1}{\rho}W^{(-\lambda_{0})}(0,u)Z^{(-\lambda_{0})}(u)m(du). \label{}
	\end{align}
	From Proposition \ref{prop:lambdaX} (iv), we see $\pi$ is a distribution on $(0,\ell)$.
	The distribution $\pi$ is a stationary distribution of $X$ under $\bP^{(\lambda_{0})}$.
	Indeed, since $\nu_{\lambda}$ is a QSD, it holds for a non-negative measurable function $f$ on $I$
	\begin{align}
		\bE_{\pi}^{(\lambda_{0})}[f(X_{t})] &= \frac{1}{\rho}\int_{(0,\ell)}W^{(-\lambda_{0})}(0,x)Z^{(-\lambda_{0})}(x) \bE_{x}\left[ \mathrm{e}^{\lambda_{0}t} \frac{Z^{-\lambda_{0}}(X_{t})}{Z^{(-\lambda_{0})}(x)} f(X_{t}) , {\tau_0}
			> t \right] {m(dx)}\label{} \\
		&= \frac{\mathrm{e}^{\lambda_{0}t}}{\lambda_{0}\rho} \bE_{\nu_{\lambda}}[Z^{(-\lambda_{0})}(X_{t})f(X_{t}), {\tau_0}> t] \label{} \\
		&= \frac{1}{\lambda_{0}\rho} \int_{(0,\ell)}Z^{(-\lambda_{0})}(u)f(u)\nu_{\lambda}(du) \label{} \\
		&= \int_{(0,\ell)}f(u)\pi(du). \label{}
	\end{align}
	Now we see the limit \eqref{meanYaglom} is the mean ergodicity w.r.t.\ the stationary distribution $\pi$:
	\begin{align}
		\lim_{t \to \infty} \bE_{x}^{(\lambda_{0})}\left[\frac{1}{t}\int_{0}^{t} f(X_{s})ds \right] 
		= \int_{(0,\ell)} f(u)\pi(du). \label{}
	\end{align}
\end{Rem}

{We finish this section by giving two independent results.
The one is a characterization of $\lambda_{0}$ by the positivity of the scale functions.
The corresponding results for one-dimensional diffusions and spectrally positive L\'evy processes have been given in \cite[Lemma 2]{Mandl} and \cite[Theorem 1.5]{QSD_SPL}, respectively.
We omit the proof since it is the same as \cite[Theorem 1.9]{QSD_downward_skip-free}.}

\begin{Thm}[cf. {\cite[Theorem 1.9]{QSD_downward_skip-free}}] \label{thm:spectralBottomCharacterization-c}
	It holds
	\begin{align}
		\lambda_{0} = \sup \{ \lambda \geq 0 \mid W^{(-\lambda)}(0,x) > 0 \ \text{for every $x \in (0,\ell)$} \}. \label{eq42-c}
	\end{align}
\end{Thm}

% \begin{proof}
% 	Set the RHS of \eqref{eq42-c} as $\tilde{\lambda}_{0}$.
% 	We see $\lambda_{0} \leq \tilde{\lambda}_{0}$ from Corollary \ref{cor:non-negativityOfScale-entrance-c}.
% 	Suppose $\tilde{\lambda}_{0} > \lambda_{0}$.
% 	For $\lambda' \in (\lambda_{0}, \tilde{\lambda}_{0})$, we see from Proposition \ref{prop:resolventIdentity} that for $q \in (-\lambda_{0},\infty) \cup \{0\}$
% 	\begin{align}
% 		h^{(q)}(\lambda';0) := \lim_{y \to\infty} \frac{W^{(-\lambda')}(0,y)}{W^{(q)}(0,y)} = 1 - (q+\lambda') \int_{(0,\ell)}W^{(-\lambda')}(0,u)g^{(q)}(u)m(du) \in [0,1). \label{eq74}
% 	\end{align}
% 	In particular, it holds $\int_{(0,\ell)}W^{(-\lambda')}(0,u)m(du) < \infty$.
% 	Let the boundary $\ell$ be non-entrance.
% 	If $h^{(q)}(\lambda';0) > 0$ for some $q \geq 0$, it follows $\int_{(0,\ell)}W^{(q)}(0,u)m(du) < \infty$ and contradicts to the non-entrance condition.
% 	Thus, $h^{(q)}(\lambda';0) = 0$ for every $q \geq 0$.
% 	Then the same computation in Lemma \ref{lem:scaleFuncGivesQSD-entrance-c} shows that $\nu_{\lambda'}(du):= \lambda' W^{(-\lambda')}(0,u) m(du)$ is a quasi-stationary distribution such that $\bP_{\nu_{\lambda'}}[\tau_{0} > t] = \mathrm{e}^{-\lambda' t}$, which is, however, impossible from the definition of $\lambda_{0}$.
% 	Next let the boundary $\ell$ be entrance.
% 	Taking limit as $q \to -\lambda_{0}+$ in \eqref{eq74}, we see from Lemma \ref{lem:r-recurrence} that the limit \eqref{eq74} is equal to $-\infty$, and it contradicts to $h^{(q)}(\lambda';0) \in [0,1)$.
% \end{proof}
{
The other is a stochastic ordering of QSDs for the non-entrance boundary case.
For distributions $\mu$ and $\nu$ on $I$, we denote $\mu \leq_{\mathrm{lr}} \nu$ when it holds for any measurable sets $A,B \subset I$ such that $\sup A \leq \inf B$
\begin{align}
	\mu(B)\nu(A) \leq \mu(A)\nu(B). \label{eq79}
\end{align}
The order is called the \textit{likelihood ratio order} (see, e.g., \cite[Section 1.C]{StochasticOrder}).
We note that the order is stronger than the usual stochastic order, that is, \eqref{eq79} implies $\mu(I_{\geq x}) \leq \nu(I_{\geq x})$ for every $x \in I$ (see \cite[Theorem 1.C.1]{StochasticOrder}).
\begin{Thm}[{cf. \cite[Theorem 1.8]{QSD_downward_skip-free}}]
	Suppose $\lambda_{0} > 0$ and the boundary $\infty$ is non-entrance.
	Then it holds
	\begin{align}
		\nu_{\lambda} \leq_{\mathrm{lr}} \nu_{\lambda'} \quad (0 < \lambda' < \lambda \leq \lambda_{0}) \label{}
	\end{align}
\end{Thm}
}
{
	\begin{proof}
		It is enough to show for every fixed $x,y \in I$ with $x < y$ that
		the function
		\begin{align}
			f(\lambda) := \frac{W^{(-\lambda)}(0,x)}{W^{(-\lambda)}(0,y)} \quad (\lambda \in (0,\lambda_{0}]) \label{}
		\end{align}
		is non-decreasing.
		From \cite[Lemma 3.5]{NobaGeneralizedScaleFunc} and its analytic extension, it holds
		\begin{align}
			f(\lambda) = \frac{1}{W^{(-\lambda)}(x,y)} \bE_{x}\left[ \int_{0-}^{\tau_{0} \wedge \tau_{y}^{+}} \mathrm{e}^{\lambda t}dL_{t}^{x} \right]. \label{}
		\end{align}
		Since obviously $\lambda_{0}^{[x,y]} \geq \lambda_{0}$,
		the RHS is non-decreasing from Proposition \ref{prop:positivityOfW} (iii).
	\end{proof}
}

\subsection{Yaglom limit for the entrance boundary case}

In this subsection, we assume the boundary $\ell$ is entrance.
Applying the $R$-theory for general Markov Processes,
we can show under some regularity assumptions the unique QSD $\nu_{\lambda_{0}}$ is the Yaglom limit {(see \eqref{eq68})}.
For the results and terminology on $R$-theory we use below, see Tuominen and Tweedie \cite{TuominenTweedie}.

Denote a transition semigroup of $X$ killed at $0$ as
\begin{align}
	P_{t}(x,du) := \bP_{x}[X_{t} \in du, \tau_{0} > t] \quad (t \geq 0,\ u \in I_{>0}, \ x \in I_{> 0}). \label{}
\end{align}
We denote by $\cB$ the set of Borel sets on $I_{>0}$.
In order to apply the $R$-theory, 
we assume the following {strong continuity of the semigroup $P_{t}$}:
for every $x \in I_{>0}$ and $A \in \cB$,
\begin{align}
	\text{the function} \ (0,\infty) \ni t  \longmapsto P_{t}(x,A) \quad \text{is continuous}. \label{continuityOfTransitionProb}
\end{align}
From Remark \ref{rem:non-negativityOfPotentialDensity}, the semigroup $P_{t}$ is $m$\textit{-irreducible}, that is, for every $x \in I_{>0}$ and $A \in \cB$ such that $m(A) > 0$,
\begin{align}
	\int_{0}^{\infty}P_{t}(x,A) dt > 0. \label{}
\end{align}
Thus, from \cite[Theorem 1]{TuominenTweedie},
the semigroup $P_{t}$ is \textit{simultaneously $m$-irreducible} (see \cite[p.787]{TuominenTweedie}).
In addition, from Lemma \ref{lem:analyticExtentionOfh}, the measure
\begin{align}
	\alpha(dx) := W(0,x) m(dx) \quad (x \in I) \label{} 
\end{align}
is a finite measure and equivalent to $m$.
Then $P_{t}$ is $\alpha$\textit{-irreducible}, and
it holds for $A \in \cB$ such that $\alpha(A) = 0 \ (\Leftrightarrow m(A) = 0 )$ that
\begin{align}
	\int_{I} P_{t}(x,A) \alpha(dx) = 0 \quad (t > 0). \label{eq58}
\end{align}
Indeed, it holds from \eqref{potentialDensityOneside-c} that
\begin{align}
	\int_{0}^{\infty}dt \int_{I} P_{t}(x,A) \alpha(dx) &= \int_{A} W(0,x) m(dx) \int_{I} (W(0,u) - W(x,u)) m(du) \label{} \\
	&\leq \alpha(A) \alpha(I) = 0. \label{}
\end{align}
Then we see \eqref{eq58} holds from \eqref{continuityOfTransitionProb}.
We have already checked that the distribution $\nu_{\lambda_{0}}$ and the function $Z^{(-\lambda_{0})}$ are $\lambda_{0}$-invariant in Theorem \ref{thm:existenceOfQSD} and Proposition \ref{prop:lambdaX}, respectively.
In addition, it holds from Proposition \ref{prop:lambdaX} that
\begin{align}
	\int_{I} Z^{(-\lambda_{0})}(u) \nu_{\lambda_{0}}(du) = \lambda_{0}\rho < \infty. \label{}
\end{align}
Thus, the semigroup $P_{t}$ is \textit{$\lambda_{0}$-positive recurrent} (see \cite[p.792]{TuominenTweedie}).
Now we can strengthen the convergence \eqref{meanYaglom} and show that $\nu_{\lambda_{0}}$ is the Yaglom limit.

\begin{Thm} \label{thm:YaglomLimit}
	The following holds:
	\begin{align}
		\lim_{t\to \infty} \sup_{A \in \cB} \left|\mathrm{e}^{\lambda_{0}t}\bP_{x}[X_{t} \in A,\tau_{0} > t] - \frac{Z^{(-\lambda_{0})}(x)}{\lambda_{0} \rho} \nu_{\lambda_{0}}(A)\right| = 0. \label{eq40}
	\end{align}
	In particular, the QSD $\nu_{\lambda_{0}}$ is the Yaglom limit {in the total variation sense}, that is,
	\begin{align}
		%\color{blue}
		\lim_{t \to \infty} \sup_{A \in \cB}|\bP_{x}[X_{t} \in A \mid \tau_{0} > t] - \nu_{\lambda_{0}}(A)| = 0. \quad (x \in I_{>0}). \label{eq68}
	\end{align}
\end{Thm}

\begin{proof}
	For $x \in I_{>0}$ and $t \geq 0$, set
	\begin{align}
		d(x,t) := \sup_{A \in \cB} \left|\mathrm{e}^{\lambda_{0}t}\bP_{x}[X_{t} \in A,\tau_{0} > t] - \frac{Z^{(-\lambda_{0})}(x)}{\lambda_{0} \rho} \nu_{\lambda_{0}}(A) \right|
	\end{align}
	From (5.9) in \cite[p.802]{TuominenTweedie}, it holds $\lim_{t \to \infty} d(x,t) = 0$ for $m$-a.e.\ $x \in I_{>0}$.
	It is enough to show that this is extended to $I_{>0}$.
	For every $y \in I_{>0}$, we can take $x \in (0,y)$ so that $\lim_{t \to \infty} d(x,t) = 0$.
	From Proposition \ref{prop:lambdaX} (ii), it holds for $A \in \cB$ and $\lambda \in (\lambda_{0}, \lambda_{0}^{[x]}(y))$
	\begin{align}
		& \sup_{A \in \cB}|\mathrm{e}^{\lambda_{0}t} \bP_{y}[X_{t} \in A, \tau_{0} > t] - \mathrm{e}^{\lambda_{0}t}\bP_{y}[X_{t} \in A, \tau_{0} > t, \tau_{x} \leq t]| \label{} \\
		\leq & \mathrm{e}^{\lambda_{0} t} \bP_{y}[\tau_{x} > t] \label{} \\
		\leq &\mathrm{e}^{-(\lambda - \lambda_{0})t} g^{(-\lambda)}(x,y) \xrightarrow[]{t \to \infty} 0, \label{eq67}
	\end{align}
	where in the last inequality we used Chebyshev's inequality.
	From \eqref{eq53} and the Markov property, we have
	\begin{align}
	%\begin{aligned}
		& \sup_{A \in \cB} \left| \mathrm{e}^{\lambda_{0}t} \bP_{y}[X_{t} \in A, \tau_{0} > t, \tau_{x} \leq t] - \frac{Z^{(-\lambda_{0})}(y)}{\lambda_{0}\rho} \nu_{\lambda_{0}}(A) \right| \label{} \\
		\leq & \int_{0}^{t} d(x,t-s) \mathrm{e}^{\lambda_{0}s}\bP_{y}[\tau_{x} \in ds] + \bE_{y}[\mathrm{e}^{\lambda_{0}\tau_{x}}, \tau_{x} > t]. 
		%\end{aligned}
		\label{eq66}
	\end{align}
	Since it clearly holds $\sup_{t \geq 0} d(x,t) < \infty$, we have from the dominated convergence theorem $\eqref{eq66} \to 0 \ (t \to \infty)$.
	Combining \eqref{eq67} and \eqref{eq66}, we obtain $\lim_{t \to \infty} d(y,t) = 0$.
	{From the triangle inequality, we easily see for $x \in I_{>0}$
	\begin{align}
		\sup_{A \in \cB}|\bP_{x}[X_{t} \in A \mid \tau_{0} > t] - \nu_{\lambda_{0}}(A) |
		\leq \frac{2d(x,t)}{\mathrm{e}^{\lambda_{0}t}\bP_{x}[\tau_{0} > t]}. \label{}
	\end{align}
	Since it holds $\lim_{t \to \infty}\mathrm{e}^{\lambda_{0}t}\bP_{x}[\tau_{0} > t] = Z^{(-\lambda_{0})}(x)/(\lambda_{0}\rho) > 0$ from \eqref{eq40},
	we obtain \eqref{eq68}.
	}
\end{proof}

% \appendix

% \section{Existence of quasi-stationary distribution: Two-side exit case}

% Suppose the boundary $\ell$ is accessible; $\lim_{y \to \ell-}\bP_{x}[\tau_{y} < \infty] > 0$.
% Set 
% \begin{align}
% 	\lambda_{0} := \sup \{ \lambda \geq 0 \mid \bE_{x}[\mathrm{e}^{\lambda (\tau_{0}\wedge \tau_{\ell})} ] < \infty \quad \text{for every $x \in (0,\ell)$} \}. \label{}
% \end{align}
% Let $f^{(q)}(x) := \lim_{z \to \ell-} \bE_{x}[\mathrm{e}^{-q\tau_{0}}, \tau_{0} < \tau_{z}] \ (q \geq 0, \ x \in [0,\ell))$.
% Then from Theorem \ref{thm:potentialDensity} we have
% \begin{align}
% 	\int_{0}^{\infty} \mathrm{e}^{-qt} \bP_{x}[X_{t} \in dy, \tau_{0} \wedge \tau_{\ell} > t]dt = v^{(q)}(x,y){m(dy)} \label{}
% \end{align}
% for
% \begin{align}
% 	v^{(q)}(x,y) := f^{(q)}(x)W^{(q)}(0,y) - W^{(q)}(x,y). \label{}
% \end{align}

% \begin{Lem}
	
% \end{Lem}

\begin{appendix}
	\section{Existence of quasi-stationary distributions: Accessible case} \label{appendix:twoSideExit}

	%{\color{blue} [I changed the name from "Two-side exit case" since I added the case where the right boundary $\ell$ is reflecting.
	%Though I changed many parts of this appendix, the changes are not essential.]}

	In this appendix, we suppose the state space $I = [0,\ell] \ (\ell \in \bR)$,
	and the boundary $\ell$ is accessible:
	\begin{align}
		\bP_{x}[\tau_{\ell} < \tau_{0}] > 0 \quad (x \in (0,\ell)). \label{}
	\end{align}
	We consider the existence of a QSD when $X$ is killed at $\tau_{0}$ or $\tau_{0} \wedge \tau_{\ell}$,
	that is, a distribution $\nu$ on $I$ such that
	\begin{align}
		\bP_{\nu}[X_{t} \in dy \mid \tau > t] = \nu(dy) \quad \text{for every $t \geq 0$} \label{}
	\end{align}
	for $\tau = \tau_{0}$ or $\tau_{0} \wedge \tau_{\ell}$.
	Though most of the results in this section can be obtained by a slight modification of the arguments for the inaccessible case discussed in Section \ref{section:QSD},
	we outline the proofs of some results for completeness.
	The two cases $\tau_{0}$ and $\tau_{0} \wedge \tau_{\ell}$ can be treated almost identically.
	Specifically, by replacing the formula \eqref{potentialDensity} with \eqref{a001} in the proof for $\tau_{0} \wedge \tau_{\ell}$,
	we can obtain the proof for $\tau_{0}$.
	Therefore, we will first outline the proof for $\tau_{0} \wedge \tau_{\ell}$, and then only state the main results which hold for $\tau_{0}$.

	Let us consider $X$ is killed at $\tau := \tau_{0}\wedge \tau_{\ell}$.
	We suppose the following irreducibility:
	\begin{align}
		\bP_{x}[\tau_{y} < \tau] > 0 \quad (x \in (0,\ell), \ y \in I). \label{eq75} 
	\end{align}
	We note that $\bP_{x}[\tau < \infty] = 1$ from Proposition \ref{prop:integrabilityOfW}.
	Set for $x \in [0,\ell)$ and $y \in (x,\ell)$
	\begin{align}
		\kappa_{0}^{[x]}(y) := \sup \{ \lambda \geq 0 \mid \bE_{y}[\mathrm{e}^{\lambda \tau_{x}^{-}},\tau_{x}^{-} < \tau_{\ell}] < \infty \} \in [0,\infty]. \label{}
	\end{align}
	As in Proposition \ref{prop:finiteLambda0}, we see that the value $\kappa_{0}^{[0]}(y)$ does not depend on $y \in (0,\ell)$ and is finite under the condition \eqref{eq75}.
	We denote it by $\kappa_{0}$.
	It is obvious that $\kappa_{0} = \lambda^{[0,\ell]}$,
	and from that and Proposition \ref{prop:positivityOfW}, we have $W^{(-\lambda)}(x,y) > 0$ for $\lambda \in (-\infty,\kappa_{0})$ and $x,y \in I$ with $x < y$.

	The following two theorems are the main results in this appendix, where the value $\rho' \in (0,\infty)$ will be defined in \eqref{eq76}.

	\begin{Thm} \label{thm:qsdExistenceTwo}
		There exists a unique QSD
		\begin{align}
			\mu_{\kappa_{0}}(dy) := C_{\kappa_{0}} W^{(-\kappa_{0})}(0,y)m(dy) \quad (y \in (0,\ell)) \label{}
		\end{align} 
		such that $\bP_{\mu_{\kappa_{0}}}[\tau > t] = \mathrm{e}^{-\kappa_{0}t} \ (t \geq 0)$,
		where $C_{\kappa_{0}} :=(\int_{(0,\ell)}W^{(-\kappa_{0})}(0,y)m(dy))^{-1}> 0$ is the normalizing constant. 
		It also holds $W^{(-\kappa_{0})}(0,x) > 0$ for $m$-a.e.\ $x \in (0,\ell)$.
	\end{Thm}	

	\begin{Thm} \label{thm:YaglomLimitTwo}
		For a function $f: (0,\ell) \to \bR$ such that $\int_{(0,\ell)}|f(u)|W^{(-\kappa_{0})}(u,\ell)m(du) < \infty$, it holds
		\begin{align}
			\lim_{t \to \infty} \frac{1}{t} \int_{0}^{t} \mathrm{e}^{\kappa_{0}s} \bE_{x}[f(X_{s}), \tau > s]ds = \frac{W^{(-\kappa_{0})}(x,\ell)}{\rho'}\int_{(0,\ell)}f(u) W^{(-\kappa_{0})}(u,\ell)m(du). \label{}
		\end{align}
		Suppose in addition that for every $x \in (0,\ell)$ and $A \in \cB(0,\ell)$, where $\cB(0,\ell)$ denotes the Borel $\sigma$-field on $(0,\ell)$, it holds that
		\begin{align}
			\text{the function} \ (0,\infty) \ni t  \longmapsto \bP_{x}[X_{t} \in A, \tau > t] \quad \text{is continuous}. \label{continuityOfTransitionProb-two}
		\end{align}
		Then we have the following convergence:
		\begin{align}
			\lim_{t\to \infty} \sup_{A \in \cB(0,\ell)} \left|\mathrm{e}^{\kappa_{0}t}\bP_{x}[X_{t} \in A,\tau > t] - \frac{W^{(-\kappa_{0})}(x,\ell)}{C_{\kappa_{0}}\rho'} \mu_{\kappa_{0}}(A)\right| = 0. \label{eq40-two}
		\end{align}
		In particular, the QSD $\mu_{\kappa_{0}}$ is the Yaglom limit in the total variation sense, that is,
		\begin{align}
			\lim_{t \to \infty} \sup_{A \in \cB(0,\ell)}|\bP_{x}[X_{t} \in A \mid \tau > t] - \mu_{\kappa_{0}}(A)| = 0. \quad (x \in (0,\ell)). \label{eq68-two}
		\end{align}
	\end{Thm}
	
	For the proof, we show a statement corresponding to Lemma \ref{lem:r-recurrence}.

	\begin{Lem} \label{lem:kappa0Characterization}
		For every $x,y \in [0,\ell)$ with $x < y$, it holds $\kappa_{0}^{[x]}(y) > 0$. 
		If $\kappa_{0}^{[x]}(y) < \infty$,
		the function $\bC \ni q \mapsto W^{(q)}(y,\ell) / W^{(q)}(x,\ell)$ has a pole at $q = -\kappa_{0}^{[x]}(y)$,
		and it is the one with the minimum absolute value.
		In particular, it holds
		\begin{align}
			\lim_{q \to -\kappa_{0}^{[x]}(y)+} \bE_{y}[\mathrm{e}^{-q \tau_{x}^{-}}, \tau_{x}^{-} < \tau_{\ell} ] = \infty. \label{}
		\end{align}
	\end{Lem}

	\begin{proof}
		Since it holds from Proposition \ref{prop:exitProblem1} that 
		\begin{align}
			\bE_{y}[\mathrm{e}^{-q\tau_{x}^{-}}, \tau_{x}^{-} < \tau_{\ell}] = \frac{W^{(q)}(y,\ell)}{W^{(q)}(x,\ell)} \quad (q > -\kappa_{0}^{[x]}(y)), \label{}
		\end{align}
		by replacing $Z^{(q)}$ with $W^{(q)}(\cdot,\ell)$ in the proof of \cite[Lemma 4.9]{QSD_downward_skip-free}, we can show the desired result.
		% From Theorem \ref{thm:potentialDensity}, we have for $q \geq 0$, $x \in [0,\ell)$ and $y \in (x,\ell)$
		% \begin{align}
		% 	&\bE_{y}[\mathrm{e}^{-q(\tau_{x}^{-} \wedge \tau_{\ell})}] \label{} \\
		% 	= &1 - q \left(\frac{W^{(q)}(y,\ell)}{W^{(q)}(x,\ell)} \int_{(x,\ell)}W^{(q)}(x,u)m(du) - \int_{(y,\ell)} W^{(q)}(y,u)m(du)\right). \label{}
		% \end{align}
		% From Corollary \ref{cor:ratioLimitScaleFunc} and Proposition \ref{prop:positivityOfW},
		% we have for $q \in [-\kappa_{0}^{[x]}, \infty) \cap \bR$
		% \begin{align}
		% 	\int_{(x,\ell)} W^{(q)}(x,u)m(du) \geq \int_{(x,\ell)} W^{(-\kappa_{0}^{[x]})}(x,u)m(du) > 0. \label{}
		% \end{align}
		% By replacing $Z^{(-\zeta)}(x)$ with $W^{(-\zeta)}(x,\ell)$ in the proof of Lemma \ref{lem:r-recurrence}, we can show the desired result.
	\end{proof}

	By the case $x = 0$, we have the following, whose proof is essentially the same as Corollary \ref{cor:lambdaZeroPositivity} and we omit it.

	\begin{Cor} \label{cor:kappaZeroPositivity}
		It holds $\kappa_{0} \in (0,\infty)$ and the process $X_{\cdot \wedge \tau}$ is $\kappa_{0}$-recurrent, that is, for every measurable set $A \subset (0,\ell)$ with $m(A) > 0$
		\begin{align}
			\int_{0}^{\infty}\mathrm{e}^{\kappa_{0}t}\bP_{x}[X_{t} \in A, \tau > t]dt = \infty \quad (x \in (0,\ell)). \label{lambdaRecurrence-two}
		\end{align}
		In addition, we have the following characterization of $\kappa_{0}$:
		\begin{align}
			\kappa_{0} = \min \{ \lambda \geq 0 \mid W^{(-\lambda)}(0,\ell) = 0 \}. \label{eq45-two}
		\end{align}
	\end{Cor}

	We prove Theorem \ref{thm:qsdExistenceTwo}.
	
	\begin{proof}[Proof of Theorem \ref{thm:qsdExistenceTwo}]
		From the same argument in Lemma \ref{lem:qsdCharacterization-c},
		if there exists a QSD $\mu$ such that $\bP_{\mu}[\tau > t] = \mathrm{e}^{-\kappa t} \ (\kappa \in (0,\kappa_{0}])$, it must be of the form $\mu(dx) = C W^{(-\kappa)}(0,x)m(dx)$ for the normalizing constant $C > 0$.
		From the same argument in Lemma \ref{lem:scaleFuncGivesQSD-entrance-c} we see that the distribution of that form is a QSD if and only if $W^{(-\kappa)}(0,\ell) = 0$.
		Thus, from Corollary \ref{cor:kappaZeroPositivity} there is a unique QSD $\mu_{\kappa_{0}}$.
		Since $\mu_{\kappa_{0}}$ is a QSD with $\bP_{\mu_{\kappa_{0}}}[\tau > t] = \mathrm{e}^{-\kappa_{0} t}$, we have from Proposition \ref{prop:exitProblem1} and Theorem \ref{thm:potentialDensity} that for $A \subset (0,\ell)$ with $m(A) > 0$
		\begin{align}
			&\frac{1}{\kappa_{0}} \int_{A}W^{(-\kappa_{0})}(0,x)m(dx) \label{} \\
			= &\int_{0}^{\infty}dt \int_{(0,\ell)} W^{(-\kappa_{0})}(0,x) \bP_{x}[X_{t} \in A, \tau > t] m(dx) \label{} \\
			= &\int_{A} W(0,u) m(du) \int_{(0,\ell)} W^{(-\kappa_{0})}(0,x) \bP_{x}[\tau_{u}^{+} < \tau_{0} < \tau_{\ell}]m(dx) > 0,
		\end{align} 
		where we note that $\bP_{x}[\tau_{u}^{+} < \tau_{0} < \tau_{\ell}] > 0 \ (x \in (0,\ell))$ by the assumption \eqref{eq75}.
		Thus, it holds $W^{(-\kappa_{0})}(0,x) > 0$ $m$-a.e.\ $x \in (0,\ell)$.
	\end{proof}

	To prove Theorem \ref{thm:YaglomLimitTwo}, we give a lemma corresponding to Proposition \ref{prop:lambdaX}.

	\begin{Lem} \label{lem:twoSideExit}
		The following holds:
		\begin{enumerate}
			\item The function $W^{(-\kappa_{0})}(\cdot,\ell)$ is $\kappa_{0}$-invariant, that is, $W^{(-\kappa_{0})}(x,\ell) > 0 \ (x \in (0,\ell))$ and
			\begin{align}
				\bE_{x}[W^{(-\kappa_{0})}(X_{t},\ell), \tau > t] = \mathrm{e}^{-\kappa_{0} t}W^{(-\kappa_{0})}(x,\ell) \quad (t \geq 0, \ x \in (0,\ell)). \label{eq63-c}
			\end{align}
			\item It holds $\kappa_{0} < \kappa_{0}^{[x]}$ for every $x \in (0,\ell)$.
			\item It holds
			\begin{align}
				\rho' := \left.\frac{d}{dq} W^{(q)}(0,\ell)\right|_{q = -\kappa_{0}} = W^{(-\kappa_{0})} \otimes W^{(-\kappa_{0})}(0,\ell) \in (0,\infty). \label{eq76}
			\end{align}
			In particular, the function $q \mapsto W^{(q)}(0, \ell)$ has a simple root at $q = -\kappa_{0}$.
		\end{enumerate}
	\end{Lem}

	\begin{proof}
		Since it holds $W^{(-\kappa_{0})}(\ell - \delta,\ell) > 0$ for every sufficiently small $\delta > 0$ from Corollary \ref{cor:ratioLimitScaleFunc}, the assertion (i) follows from the proof of Proposition \ref{prop:lambdaX} with $Z^{(q)}$ replaced by $W^{(q)}(\cdot,\ell)$.
		The assertion (ii) follows from Lemma \ref{lem:kappa0Characterization} and (i).
		Since it holds from Proposition \ref{prop:resolventIdentity} and (i)
		\begin{align}
			\left. \frac{d}{dq}W^{(q)}(0,\ell)\right|_{q = -\kappa_{0}} = W^{(-\kappa_{0})} \otimes W^{(-\kappa_{0})}(0,\ell) > 0. \label{}
		\end{align}
		The proof is complete.
	\end{proof}

	With the preparation above, the proof of Theorem \ref{thm:YaglomLimitTwo} is essentially the same as Theorems \ref{thm:meanYaglomLimit} and \ref{thm:YaglomLimit}, and we omit it. 

	We next consider the case where $X$ is killed at $\tau_{0}$.
	Hereafter, we suppose the following:
	\begin{align}
		\bP_{x}[\tau_{\ell} < \tau_{0}] > 0 \quad (x \in I_{>0}). \label{}
	\end{align}
	We note that from \eqref{a001}, it holds
	\begin{align}
		g^{(q)}(x) := \bE_{x}[\mathrm{e}^{-q\tau_{0}}] = \frac{Z^{(q)}(x,\ell)}{Z^{(q)}(0,\ell)} \quad (q \geq 0,\ x \in I_{>0}), \label{}
	\end{align}
	and the rightmost term is meromorphic on $\bC$.
	Thus, the arguments above related to the analyticity still work in this situation.
	Since the all the results in this appendix shown for $\tau_{0} \wedge \tau_{\ell}$ still holds for $\tau_{0}$ with an obvious modification, 
	we present here only the results regarding the existence of QSDs and omit the rest.
	\begin{Thm}
		Let 
		\begin{align}
			\iota_{0} := \sup \{ \lambda \geq 0 \mid \bE_{x}[\mathrm{e}^{\lambda \tau_{0}}] < \infty \quad \text{for some (or every) } x \in I_{>0} \}. \label{}
		\end{align}
		There exists a unique QSD
		\begin{align}
			\mu_{\iota_{0}}(dy) := \iota_{0} W^{(-\iota_{0})}(0,y)m(dy) \quad (y \in I_{>0}) \label{}
		\end{align} 
		such that $\bP_{\mu_{\iota_{0}}}[\tau_{0} > t] = \mathrm{e}^{-\iota_{0}t} \ (t \geq 0)$. 
		It also holds $W^{(-\iota_{0})}(0,x) > 0$ for every $x \in I_{>0}$.
	\end{Thm}

\section{The proofs of some propositions} \label{AppA01}
In this section, we put the proofs of some propositions used in the main part of this paper. 

\subsection{The proof of Proposition \ref{Prop101}}
We fix $x, y \in I$ with $x<y$. 
\par
By the definition of $m$, it is enough to prove that $R^{(0)}1_{(x, y)} (y) >0$. 
It holds 
\begin{align}
R^{(0)}1_{(x, y)} (y)=\bE_y \sbra{\int_0^\infty 1_{(x, y)}  (X_t) dt}\geq \bE_y \sbra{\int_{\tau^{\leq x}_y}^{\tau^-_x} 1_{(x, y)}  (X_t) dt}, \label{A01}
\end{align}
where $\tau^{\leq x}_y:= \sup \{ t \geq 0 \mid X_{t} = y, \ X_{s} < y \text{ for } s \in (t,\tau_{x}^{-}]\}$.
\par 
Fix $\omega \in \Omega$ such that $t \mapsto X_{t}(\omega)$ has a c\'adl\'ag path, does not have negative jumps and satisfies $\tau^{-}_{x}(\omega) <\infty$. 
Then, $\tau^{\leq x}_y(\omega)\leq \tau^-_x(\omega)<\infty$.
By the right-continuity of $t \mapsto X_t(\omega)$ and since $X_{\tau^{\leq x}_y}(\omega)=y $, 
there exists $\varepsilon(\omega)>0$ such that $X_{s}(\omega) > x$ for $s \in [\tau^{\leq x}_y, \tau^{\leq x}_y+ \varepsilon(\omega)) $. 
This fact implies that
\begin{align}
\int_{\tau^{\leq x}_y(\omega)}^{\tau^-_x(\omega)} 1_{(x, y)}  (X_t(\omega)) dt\geq  \varepsilon (\omega)
\end{align}
The set consisting of all such $\omega$ has strictly positive probability for $\bP_y$ by (A2), 
the last term of \eqref{A01} is strictly positive.

\subsection{The proof of Proposition \ref{Prop202}}
	Let $q > 0$.
	Recall that from the definition of the excursion measure, it holds $n_{y}[1 - \mathrm{e}^{-q \zeta}] < \infty$ for the lifetime $\zeta$ of the excursion.
	It holds from the Markov property
	\begin{align}
%	\begin{aligned}
		\infty > n_{y}[1 - \mathrm{e}^{-q\zeta}] \geq &n_{y}[\mathrm{e}^{-q\tau_{x}^{-}} - \mathrm{e}^{-q\zeta}, \tau_{x}^{-} < \zeta] 
		\\
		 =&n_{y}[\mathrm{e}^{-q\tau_{x}^{-}}, \tau_{x}^{-} < \infty ] \bE_{x}[1 - \mathrm{e}^{-q \tau_y}]. 
		 %\end{aligned}
		 \label{205}
	\end{align}
	The last term of \eqref{205} is positive by \cite[Lemma 3.5]{NobaGeneralizedScaleFunc} and (A2). 
	Therefore, by the definition of $W^{(q)}$, it holds $W^{(q)}(x,y) \in (0, \infty)$. 
	
	% If $W^{(q)}(x,y) = \infty$,
	% from the absence of negative jumps it holds for any $u \in [x,y)$
	% \begin{align}
	% 	0= n_{y}[\mathrm{e}^{-q\tau_{x}^{-}}, \tau_{x}^{-} < \infty] = n_{y}[\mathrm{e}^{-q\tau_{u}^{-}}, \tau_{u}^{-} < \infty] \bE_{u}[\mathrm{e}^{-q\tau_{x}}, \tau_{x} < \tau_{y}^{+}]. \label{}
	% \end{align}
	% Thus, we see $n_{y}[\mathrm{e}^{-q\tau_{u}^{-}}, \tau_{u}^{-} < \infty] = 0$.
	% From the monotone convergence theorem we have
	% \begin{align}
	% 	n_{y}[\mathrm{e}^{-q\tau_{{I_{<y}} }}, \tau_{{I_{<y}}} < \infty] = \lim_{u \to y-}n_{y}[\mathrm{e}^{-q\tau_{u}^{-}}, \tau_{u}^{-} < \infty] = 0, \label{}
	% \end{align}
	% which contradicts to Remark \ref{rem:downwardRegularity}.
	% Thus, we have $W^{(q)}(x,y) < \infty$ for $q > 0$.
	% Since it holds
	% \begin{align}
	% 	\infty > n_{y}[\mathrm{e}^{-q\tau_{x}}, \tau_{x} < \infty] \geq \mathrm{e}^{-q}n_{y}[\tau_{x} < 1] \label{}
	% \end{align}
	% and $n_{y}[1 \leq \tau_{x} <\infty] \leq n_{y}[\zeta \geq 1] < \infty$,
	% we see the case $q = 0$.

\subsection{The proof of Proposition \ref{prop:continuityOfW}}
	The left continuity is obtained by the following: for $x \in I \cap (- \infty , y)$, 
	\begin{align}
	\lim_{\varepsilon \to 0+}W^{(q)}(x- \varepsilon, y)=&\lim_{\varepsilon \to 0+} \frac{1}{n_y \sbra{e^{-q\tau^-_{x-\varepsilon}}, \tau^-_{x-\varepsilon}<\infty }}\\
	=&\lim_{\varepsilon \to 0+} \frac{1}{n_y \sbra{e^{-q\tau^-_{x}}, \tau^-_{x}<\infty }\bE_x \sbra{e^{-q\tau^-_{x -\varepsilon}},\tau_{x -\varepsilon} <\tau^+_y}}\\
	=&\frac{1}{n_y \sbra{e^{-q\tau^-_{x}}, \tau_{x}<\infty }}, 
	\end{align}
	where we used the strong Markov property at $\tau_x$ in the second equality and Remark \ref{rem:downwardRegularity} in the last equality.
	%By \eqref{scaleFunc} and the dominated convergence theorem, it is sufficient to check $\lim_{\varepsilon\to0}\tau_{x+\varepsilon}=\tau_x$, $n_y$-a.e., for $x \in I \cap (- \infty , y)$. 
	%By \cite[Remark 3.2]{NobaGeneralizedScaleFunc}, it holds $\lim_{\varepsilon\to0+}\tau_{x+\varepsilon}=\tau_x$, $n_y$-a.e.. 
	The inequality $\tau_{x+}:=\lim_{\varepsilon\to0+}\tau_{x+\varepsilon}^{-} \leq {\tau^-_x}$ holds for $n_y$-a.e. $x$, since the map $z \mapsto \tau_z$ is non-increasing. 
	Thus, we have
	\begin{align}
		W^{(q)}(x,y)=&\frac{1}{n_y \sbra{e^{-q{\tau^-_{x}}},{\tau^-_{x}}<\infty }} \label{} \\
		=&\frac{1}{n_y \sbra{e^{-q\tau_{x+}}\bE_{X_{\tau_{x+}}}\sbra{e^{-q{\tau^-_x}}, {\tau^-_x}<\tau^+_y } ,\tau_{x+}<\infty }}
		\\
		=&\frac{1}{n_y \sbra{e^{-q\tau_{x+}}\bE_{x}\sbra{e^{-q{\tau^-_x}}, {\tau^-_x}<\tau^+_y } ,\tau_{x+}<\infty }} \label{eq01-t} \\
		=&\frac{1}{n_y \sbra{e^{-q\tau_{x+}},\tau_{x+}<\infty }}
		\\
		=&\lim_{\varepsilon \to 0+} \frac{1}{n_y \sbra{e^{-q\tau^-_{x+\varepsilon}}, \tau^-_{x+\varepsilon}<\infty }} \label{} \\
		=&\lim_{\varepsilon \to 0+}W^{(q)}(x+ \varepsilon, y), 
	\end{align}
	where we used the quasi-left continuity of $X$ in the third equality and we used Remark \ref{rem:downwardRegularity} in the forth equality.

\end{appendix}

\bibliographystyle{plain}
\bibliography{SNQSD_2-arXiv.bbl}

\end{document}